\documentclass[12pt]{amsart}
  \usepackage{amssymb}         
  \usepackage{amsmath}          
  \usepackage{amsfonts}           
  \usepackage{amsthm}
  \usepackage{amscd}
  \usepackage{enumerate}
  \usepackage[active]{srcltx}
  \usepackage[percent]{overpic}
  \usepackage[mathscr]{eucal}
  \usepackage{graphicx,epsfig} 
  \usepackage{tikz-cd}        
  \usepackage{color}
  \usepackage{marginnote}
  \usepackage{hyperref}
  \usepackage{enumitem} 

  \def\Acc{{\rm Acc}}
  \def\C{\mathbb{C}}
  \def\Chat{\widehat{\mathbb{C}}}
  \renewcommand*{\d}{\ensuremath{{\operatorname{dist}}}}

  \def\Dbar{{\overline{\mathbb{D}}}}
  \newcommand*{\e}{\ensuremath{{\operatorname{e}}}}

  \def\Ga{\Gamma}
  \newcommand*{\mapfromto}[3]{\hbox{\ensuremath{#1 : #2 \to #3}}}
  
  \def\om{\omega}
  \def\Om{\Omega}
  \def\La{\Lambda}  
  \def\RR{{\mathbb{R}}}

  \def\TT{{\mathbb{T}}}
  \def\whd{{\widehat{d}}}
  \def\whD{{\widehat{D}}}

  \def\ZZ{{\mathbb{Z}}}
  \def\XX{{\mathcal{X}}}
  \def\YY{{\mathcal{Y}}}

\newtheorem{theorem}{Theorem}[section]
\newtheorem{corollary}[theorem]{Corollary}
\newtheorem{lemma}[theorem]{Lemma}
\newtheorem{proposition}[theorem]{Proposition}

\theoremstyle{definition}
\newtheorem{definition}[theorem]{Definition}
\newtheorem{example}[theorem]{Example}
\theoremstyle{remark}
\newtheorem{remark}[theorem]{Remark}
\theoremstyle{corx}
\newtheorem*{corx}{Corollary}
\newtheorem{conjecture}{Conjecture}
\theoremstyle{thmx}
\newtheorem{thmx}{Theorem}

\newcommand{\ALIGN}{\begin{align*}}
\newcommand{\ENDALIGN}{\end{align*}}
\newcommand{\ENUM}{\begin{enumerate}}
\newcommand{\ENUMa}{\begin{enumerate}[a.]}
\newcommand{\ENUMA}{\begin{enumerate}[A.]}
\newcommand{\ENUMAF}{\begin{enumerate}[\bf A.]}
\newcommand{\ENUMi}{\begin{enumerate}[i)]}
\newcommand{\ENDENUM}{\end{enumerate}}
\newcommand{\ITMZ}{\begin{itemize}}
\newcommand{\ENDITMZ}{\end{itemize}}
\newcommand{\REFEQN}[1] { \begin{equation}\label{#1} }
\newcommand{\ENDEQN}{\end{equation}}
\newcommand{\THM}{\begin{theorem}}
\newcommand{\REFEXA}[1] { \begin{example}\label{#1} }
\newcommand{\ENDEXA}{\end{example}}
\newcommand{\MTX}{ \begin{matrix}}
\newcommand{\ENDMTX}{ \end{matrix}}
\newcommand{\REM}{ \begin{remark}}
\newcommand{\ENDREM}{\end{remark}}
\newcommand{\REFTHM}[1] { \begin{theorem}\label{#1} }
\newcommand{\RREFTHM}[2] { \begin{theorem}[#1]\label{#2} }
\newcommand{\ENDTHM}{\end{theorem}}
\newcommand{\REFNTH}[1] { \begin{newthm}\label{#1} }
\newcommand{\ENDNTH}{\end{newthm}}
\newcommand{\REFPROP}[1]{\begin{proposition}\label{#1} }
\newcommand{\RREFPROP}[2]{\begin{proposition}[#1]\label{#2} }
\newcommand{\PROP}{\begin{proposition}}
\newcommand{\ENDPROP}{\end{proposition} }
\newcommand{\REFDEF}[1]{\begin{definition}\label{#1} }
\newcommand{\DEF}{\begin{definition}}
\newcommand{\ENDDEF}{\end{definition} }
\newcommand{\REFLEM}[1]{\begin{lemma}\label{#1} }
\newcommand{\RREFLEM}[2]{\begin{lemma}[#1]\label{#2} }
\newcommand{\LEM}{\begin{lemma}}
\newcommand{\ENDLEM}{\end{lemma} }
\newcommand{\REFCOR}[1]{\begin{corollary}\label{#1} }
\newcommand{\COR}{\begin{corollary}}
\newcommand{\ENDCOR}{\end{corollary} }
\newcommand{\CONJ}{\begin{conjecture}}
\newcommand{\REFCONJ}[1]{\begin{conjecture}\label{#1}}
\newcommand{\RREFCONJ}[2]{\begin{conjecture}{#1}\label{#2}}
\newcommand{\ENDCONJ}{\end{conjecture} }
\newcommand{\REFDEFTHM}[1] { \begin{defthm}\label{#1} }
\newcommand{\ENDDEFTHM}{\end{defthm}}

\newcommand{\corref}[1]{Corollary~\ref{#1}}

\newcommand{\figref}[1]{Fig.~\ref{#1}}
\newcommand{\lemref}[1]{Lemma~\ref{#1}}
\newcommand{\remref}[1]{Remark~\ref{#1}}
\newcommand{\thmref}[1]{Theorem~\ref{#1}}

\newcommand{\PROOF}{\begin{proof}}
\newcommand{\ENDPROOF}{\end{proof}}

\numberwithin{equation}{section}
\newcommand{\vs}{\vspace{6pt}}
\newcommand{\bit}{\it \bfseries}
\newcommand{\es}{\emptyset}
\newcommand{\sm}{\smallsetminus}
\newcommand{\ov}{\overline}
\newcommand{\bd}{\partial}
\newcommand{\ve}{\varepsilon}
\newcommand{\modd}{\ (\operatorname{mod} 1)}
\def\CC{\mathbb{C}}
\def\DD{\mathbb{D}}

\begin{document}

\title[External Rays Under Renormalization]{On the correspondence of external rays under renormalization}

\author[C. L. Petersen and S. Zakeri]{Carsten Lunde Petersen and Saeed Zakeri}

\address{Department of Mathematics, Roskilde University, DK-4000 Roskilde, Denmark} 

\email{lunde@ruc.dk}

\address{Department of Mathematics, Queens College of CUNY, 65-30 Kissena Blvd., Queens, New York 11367, USA} 
\address{The Graduate Center of CUNY, 365 Fifth Ave., New York, NY 10016, USA}

\email{saeed.zakeri@qc.cuny.edu}

\date{February 28, 2019}

\maketitle

\begin{abstract}
Let $P$ be a monic polynomial of degree $D \geq 3$ whose filled Julia set $K_P$ has a non-degenerate periodic component $K$ of period $k \geq 1$ and renormalization degree $2 \leq d<D$. Let $I=I_K$ denote the set of angles $\theta$ on the circle $\TT=\RR/\ZZ$ for which the (smooth or broken) external ray $R^P_\theta$ for $P$ accumulates on $\bd K$. We prove the following: \vs

\begin{enumerate}[leftmargin=*]
\item[$\bullet$]
$I$ is a compact set of Hausdorff dimension $<1$ and there is an essentially unique degree $1$ monotone map $\Pi: I \to \TT$ which semiconjugates $\theta \mapsto D^k \theta \modd$ on $I$ to $\theta \mapsto d \theta \modd$ on $\TT$. \vs

\item[$\bullet$] 
Any hybrid conjugacy $\varphi$ between a renormalization of $P^{\circ k}$ 
on a neighborhood of $K$ and a monic degree $d$ polynomial $Q$ 
induces a semiconjugacy $\Pi: I \to \TT$ with the property that 
for every $\theta \in I$ the external ray $R^P_\theta$ has the same accumulation set as the curve $\varphi^{-1}(R^Q_{\Pi(\theta)})$. 
In particular, $R^P_\theta$ lands at $z \in \bd K$ if and only if $R^Q_{\Pi(\theta)}$ lands at $\varphi(z) \in \bd K_Q$. \vs

\item[$\bullet$] 
The ray correspondence established by the above result is finite-to-one. In fact, the cardinality of each fiber of $\Pi$ is $\leq D-d+2$, and the inequality is strict when the component $K$ has period $k=1$. \vs
\end{enumerate}

Using a new type of quasiconformal surgery we construct a class of examples with $k=1$ for which the upper bound $D-d+1$ is realized and the set $I$ has isolated points.     
\end{abstract}

\tableofcontents

\section{Introduction}\label{sec:intro}

This paper provides an understanding of the external rays that accumulate (and in particular land) on a non-degenerate periodic component of a disconnected polynomial Julia set. It can be viewed as a complement to the well-studied case of connected Julia sets initiated by Douady and Hubbard in the late 1980's. \vs

Here is a brief outline of the results; we will give the precise definitions in \S \ref{sec:prelim} and the proofs in \S \ref{sec:pfsa}-\S \ref{sec:ex}. Let $P: \CC \to \CC$ be a monic polynomial of degree $D \geq 3$ whose filled Julia set $K_P$ is disconnected. Let $K$ be a periodic component of $K_P$ of period $k \geq 1$ which is non-degenerate in the sense that it is not a single point. There are topological disks $U_1,U_0$ containing $K$ such that the restriction $P^{\circ k}|_{U_1}: U_1 \to U_0$ is a polynomial-like map of some degree $2 \leq d<D$ with connected filled Julia set $K$. This restriction is hybrid equivalent to a polynomial $Q$ of degree $d$, that is, there is a quasiconformal map $\varphi : U_0 \to \varphi(U_0)$ which satisfies $\varphi \circ P^{\circ k} = Q \circ \varphi$ in $U_1$ and has the property that $\bar{\bd} \varphi=0$ a.e. on $K$. It follows that $\varphi(K)$ is the filled Julia set $K_Q$. The main objective of this paper is to relate the external rays of $Q$ to the external rays of $P$ which accumulate on $\bd K$. Since $K_Q$ is connected, every external ray of $Q$ is a smooth curve. For $P$, however, we need to allow all external rays, including the ones that crash into the escaping precritical points of $P$. These generalized rays can be defined in terms of the gradient flow of the Green's function $G$ of $P$ in $\CC \sm K_P$. They consist of smooth field lines of $\nabla G$ that descend from $\infty$ and approach $K_P$, as well as their limits which are broken rays that abruptly turn when they crash into a critical point of $G$. For each $\theta \in \TT :=\RR/\ZZ$ we have either a smooth ray $R_\theta$ or a pair $R_\theta^\pm$ of broken rays which descend from $\infty$ at the angle $\theta$. Here $R_\theta^+$ (resp. $R_\theta^-$) makes a right (resp. left) turn at each critical point it crashes into (see \S \ref{sec:prelim} for details). Denoting by {\mapfromto {\whD, \whd} \TT \TT} the maps 
$$
\whD(\theta) = D \, \theta, \quad \whd(\theta) = d \, \theta \ \modd,
$$
we have $P(R_\theta)=R_{\whD(\theta)}$ and $P(R_\theta^\pm)=R_{\whD(\theta)}^\pm$ or $R_{\whD(\theta)}$.   

\begin{thmx}[External angles associated with $K$] \label{A}
The set $I=I_K \subset \TT$ of angles $\theta$ for which the smooth ray $R_\theta$ or one of the broken rays $R_\theta^\pm$ accumulates on $\bd K$ is compact, invariant under $\whD^{\circ k}$ and of Hausdorff dimension $\leq \log d \, / (k \log D)$. Moreover, there is a continuous degree $1$ monotone surjection {\mapfromto \Pi I \TT} which makes the following diagram commute:
\begin{equation}\label{scj}
\begin{tikzcd}[column sep=small]
I \arrow[d,swap,"\Pi"] \arrow[rr,"\whD^{\circ k}"] & & I \arrow[d,"\Pi"] \\
\TT \arrow[rr,"\whd"] & & \TT 
\end{tikzcd} 
\end{equation}
The semiconjugacy $\Pi$ is unique up to postcomposition with a rotation of the form $\tau \mapsto \tau + j/(d-1) \modd$.
\end{thmx}

Observe that the existence of the semiconjugacy $\Pi$ implies that $I$ is uncountable and in fact contains a Cantor set. It would be tempting to speculate that $I$ itself is a Cantor set, but in \S \ref{sec:ex} we construct examples for which $I$ has isolated points (compare \thmref{D} below). \vs   

Let us denote the accumulation set of a (smooth or broken) ray $R$ by $\Acc(R)$:
$$
\Acc(R) := \overline{R} \sm R 
$$

\begin{thmx}[Ray correspondence] \label{B}
For any hybrid conjugacy $\varphi : U_0 \to \varphi(U_0)$ between the restriction $P^{\circ k}|_{U_0}: U_1 \to U_0$ and a degree $d$ monic polynomial $Q$, there is a choice of the semiconjugacy {\mapfromto \Pi I \TT} of \thmref{A} such that 
$$ 
\varphi(\Acc(R^P_\theta)) = \Acc(R^Q_{\Pi(\theta)}) \qquad \text{whenever} \ \theta \in I.
$$
In particular, $R^P_\theta$ lands at $z \in \bd K$ if and only if $R^Q_{\Pi(\theta)}$ lands at $\varphi(z) \in \bd K_Q$.
\end{thmx} 

Here $R^P_\theta$ and $R^Q_\theta$ denote the external rays at angle $\theta$ for $P$ and $Q$, respectively (in the case of a broken ray for $P$, precisely one of the two possible rays at angle $\theta \in I$ accumulates on $\bd K$ and $R^P_\theta$ denotes that choice). \vs

Note that for each $\tau \in \TT$ the preimage $\varphi^{-1}(R_\tau^Q)$ is an arc in $\CC \sm K$ that accumulates on $\bd K$ but possibly meets infinitely many components of $K_P$ along the way. The main ingredient of the proof of \thmref{B} is to show that for every $\theta \in \Pi^{-1}(\tau)$, the ray segment $R_\theta^P \cap U_1$ and the arc $\varphi^{-1}(R_\tau^Q) \cap U_1$ stay at a bounded distance in the hyperbolic metric of $U_0 \sm K$. \vs

To illustrate the content of \thmref{B}, consider a cubic polynomial of the form $P(z)=\e^{2\pi i \theta} z + b z^2 +z^3$, where $b \in \CC$ and the rotation number $\theta$ has a continued fraction $[a_1,a_2,a_3,\ldots]$ that satisfies $\log a_n=O(\sqrt{n})$ as $n \to \infty$. For large enough $|b|$ the filled Julia set $K_P$ is disconnected but has a quadratic-like restriction hybrid equivalent to $Q(z)=\e^{2\pi i \theta} z + z^2$. It follows that the component $K$ of $K_P$ containing the fixed point $0$ is quasiconformally homeomorphic to $K_Q$. According to \cite{PZ1}, the filled Julia set $K_Q$ is locally connected. Moreover, every $w \in \bd K_Q$ is the landing point of one or two rays according as the forward orbit of $w$ misses or hits the critical point $-\e^{2\pi i \theta}/2$ of $Q$. The semiconjugacy $\Pi$ of \thmref{A} is at most $2$-to-$1$ in this case (see \thmref{C} below), so \thmref{B} implies that every point of $\bd K$ is the landing point of at least one and at most four rays for $P$. \figref{zoo} shows neighborhoods of the respective critical points in $K$ and $K_Q$ for the golden mean case $\theta=(\sqrt{5}-1)/2=[ 1,1,1,\ldots ]$. Looking at the figure on the right, it is far from obvious that the critical point at the center is accessible through an arc that avoids the uncountably many components of $K_P$. \vs

\begin{figure}[t!]
  \centering
  \includegraphics[width=0.4\textwidth]{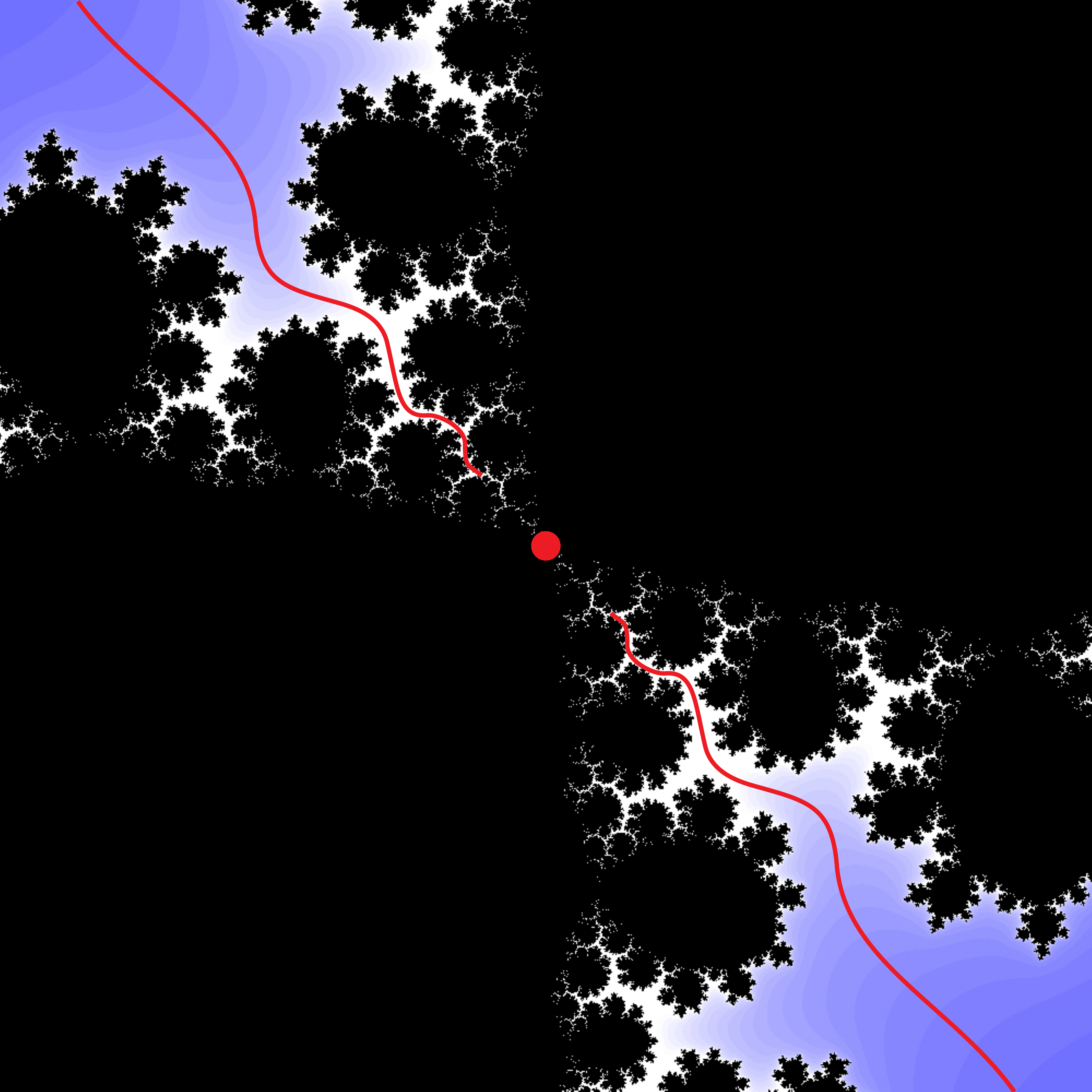}
  \hspace{1cm}
  \includegraphics[width=0.4\textwidth]{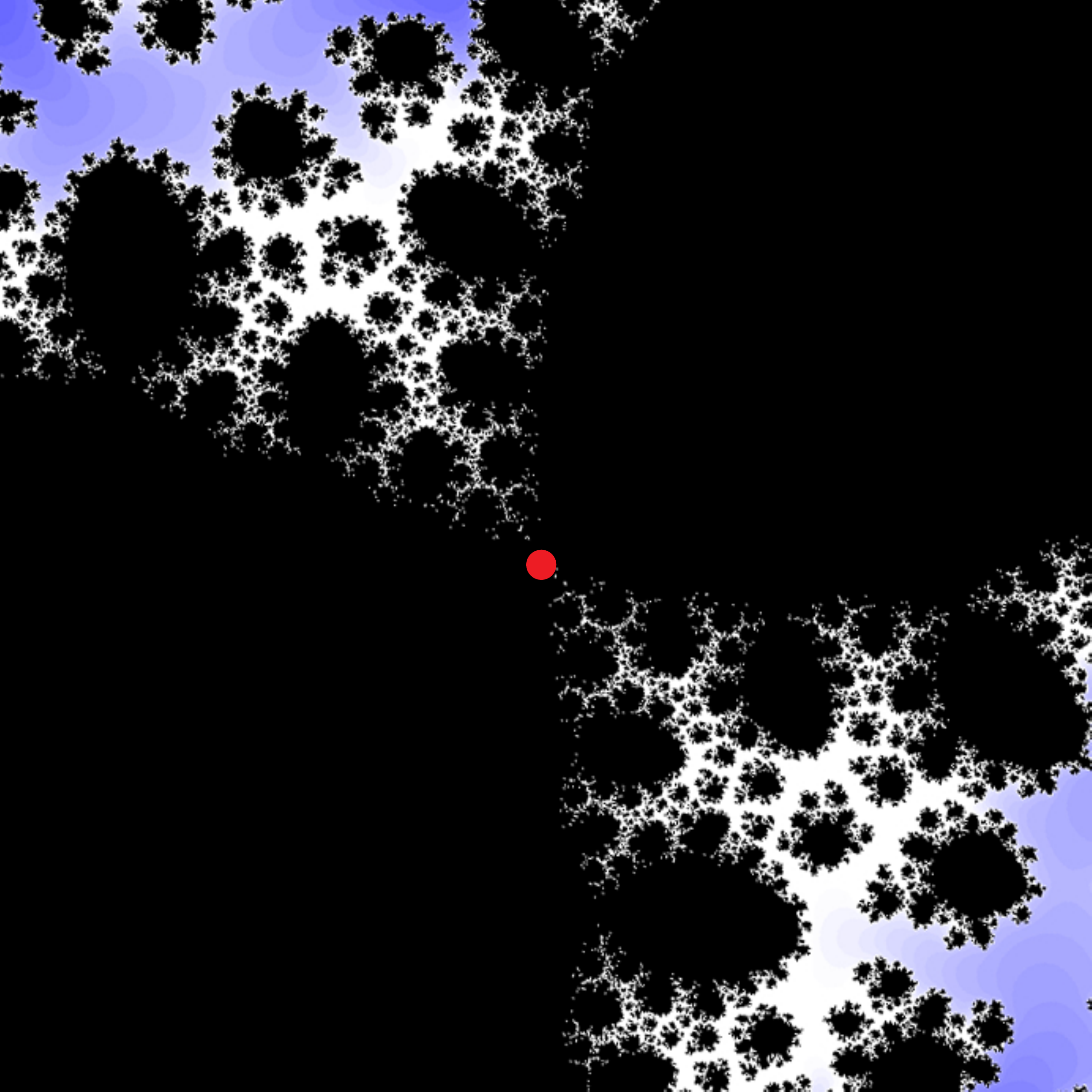}
  \caption{\sl{Left: The filled Julia set of the quadratic polynomial $Q(z)= \e^{2\pi i \theta} z + z^2$, with $\theta=(\sqrt{5}-1)/2$. Right: The disconnected filled Julia set of some cubic polynomial $P(z)= \e^{2\pi i \theta} z + b z^2 +z^3$ with a quadratic-like restriction hybrid equivalent to $Q$. Both pictures are magnified near the critical point at the center.}} 
  \label{zoo}   
\end{figure}

The following corollary of \thmref{B} is worth mentioning:

\begin{corx}
If a non-degenerate component $K$ of a polynomial filled Julia set is locally connected, every point of its boundary $\bd K$ is the landing point of at least one ray, and therefore is accessible through $\CC \sm K$.   
\end{corx}

Here the assumption of $K$ being periodic is not needed since 
every non-degenerate component of a polynomial filled Julia set is known to be eventually periodic \cite{QY}. \vs

The analog of the above corollary in the connected case is well known and in fact non-dynamical: Every boundary point of a locally connected full continuum is accessible, hence by the theorem of Lindel\"{o}f \cite{Po} it is the landing point of at least one hyperbolic geodesic in the complement descending from $\infty$. But the disconnected case asserted by the above corollary is certainly dynamical, as it fails for general compact sets (think of the union of the closed unit disk together with the semicircles $\{ (1+1/n)\e^{2\pi it}: |t|\leq 1/4 \}$ for $n \geq 1$, where every point of the right half of $\bd \DD$ is inaccessible from the complement of this union). \vs

In \S \ref{sec:val} we prove

\begin{thmx}[Valence of $\Pi$] \label{C}
The semiconjugacy $\Pi: I \to \TT$ of \thmref{A} satisfies 
$$
\sup_{\tau \in \TT} \ \# \Pi^{-1}(\tau) \leq D-d+2.
$$
The inequality is strict if the component $K$ has period $k=1$.  
\end{thmx} 

The proof has two (somewhat related) ingredients: One is the dynamics of the ``gaps'' of $I=I_K$ and its maximal Cantor subset as degree $d$ invariant sets for $\whD^{\circ k}$. The other is a bound on the cardinality of the fibers of $\Pi$ whose angles are eventually periodic under $\whD^{\circ k}$. This is related to the problem of bounding the number of cycles of smooth rays that land on a periodic point, studied in degree $2$ by Milnor \cite{M2} and extended to higher degrees by Kiwi \cite{K} in their work on ``orbit portraits'' of polynomial maps. The argument in our case is a bit more subtle, as we have to deal with periodic rays that are infinitely broken, or smooth periodic rays that pull back to pairs of broken rays. \vs    

In \S \ref{sec:ex} we present a general method for constructing examples where $K$ has period $1$ and $\Pi$ has the top valence $D-d+1$ predicted by \thmref{C}. Consider the $D-1$ fixed points 
$$
\theta_i := \frac{i}{D-1}  \modd
$$
of the map $\whD$, taking the subscript $i$ modulo $D-1$. Using the technique of quasiconformal surgery we prove the following 

\begin{thmx}[Top valence and isolated rays]\label{D}
Given integers $2 \leq d<D$, a degree $d$ polynomial $Q$ with connected filled Julia set and a fixed point $\theta_j$ of $\whD$, there is a polynomial $P$ of degree $D$ whose filled Julia set has a component $K=P(K)$ such that \vs
\begin{enumerate}
\item[(i)]
$P$ restricted to a neighborhood of $K$ is hybrid equivalent to $Q$. \vs
\item[(ii)]
The $D-d+1$ consecutive fixed points $\theta_j, \ldots, \theta_{j+D-d}$ belong to the same fiber of the semiconjugacy $\Pi: I_K \to \TT$. \vs  
\end{enumerate}
The corresponding rays $R^P_{\theta_j}, \ldots, R^P_{\theta_{j+D-d}}$ co-land at a fixed point of $P$ on $\bd K$
\end{thmx}

Compare Figs. \ref{seh} and \ref{chah}. When $D>d+1 \geq 3$, it follows that the $D-d-1$ angles $\theta_{j+1}, \ldots, \theta_{j+D-d-1}$ are isolated points of $I=I_K$. The set $I$ can have isolated points even when $D=3$ but that would require a higher period $k$ by \thmref{C} (see Example \ref{cubicex} and compare Figures \ref{ghost1} and \ref{ghost2}).   

\section{Preliminaries}\label{sec:prelim}

We assume the reader is familiar with the basic notions of complex dynamics, as in \cite{M1}. For convenience, and to establish our notations, we quickly recall some definitions. \vs 

\noindent
{\it Convention.} For distinct points $a,b \in \TT$ we use the notation $]a,b[$ for the open interval in $\TT$ traversed counterclockwise from $a$ to $b$. We define $[a,b[, ]a,b], [a,b]$ by adding the suitable endpoints to $]a,b[$.  

\subsection{Green and B\"{o}ttcher}\label{g&b}
 
Let $P:\CC \to \CC$ be a monic polynomial map of degree $D \geq 2$. The {\bit filled Julia set} $K_P$ is the union of all bounded orbits of $P$:
$$
K_P= \{ z \in \CC : \{ P^{\circ n}(z) \}_{n \geq 0} \ \text{is bounded} \}. 
$$
It is a compact non-empty subset of the plane with connected complement $\CC \sm K_P$. This complement can be described as the {\bit basin of infinity} of $P$, that is, the set of all points which escape to $\infty$ under the iterations of $P$. The {\bit Green's function}\index{Green's function} of $P$ is the continuous subharmonic function $G=G_P: \CC \to [0,+\infty[$ defined by 
$$
G(z)=\lim_{n \to \infty} \frac{1}{D^n} \log^+ |P^{\circ n}(z)|
$$
which describes the escape rate of $z$ to $\infty$ under the iterations of $P$. Here $\log^+ t = \max \{ \log t, 0 \}$. It satisfies the relation
$$
G(P(z))=D \, G(z) \qquad \text{for all} \ z \in \CC,
$$
with $G(z)=0$ if and only if $z \in K_P$. We often refer to $G(z)$ as the {\bit potential} of $z$. The Green's function is harmonic in $\CC \sm K_P$ and has critical points precisely at the escaping precritical points of $P$, that is, $\nabla G(z)=0$ for some $z \in \CC \sm K_P$ if and only if $P^{\circ n}(z)$ is a critical point of $P$ for some $n \geq 0$. It is easy to see that for every $s>0$ there are at most finitely many critical points of $G$ at potentials higher than $s$. The open set $G^{-1}([0,s[)$ has finitely many connected components, all being Jordan domains with piecewise analytic boundaries. \vs

There is a unique conformal isomorphism $\frak{B}=\frak{B}_P$, defined in some neighborhood of $\infty$, which is tangent to the identity at $\infty$ (in the sense that $\lim_{z \to \infty} \frak{B}(z)/z = 1$) and conjugates $P$ to the power map $w \mapsto w^D$: 
\begin{equation}\label{bfe}
\frak{B}(P(z))=(\frak{B}(z))^D \qquad \text{for large} \ |z|.  
\end{equation}
We call $\frak{B}$ the {\bit B\"{o}ttcher coordinate} of $P$ near $\infty$. The modulus of $\frak{B}$ is related to the Green's function by the relation 
$$
\log |\frak{B}(z)|=G(z) \qquad \text{for large} \ |z|. 
$$
Set
\begin{align*}
s_{\max} & := \max \big\{ G(c): c \ \text{is a critical point of} \ P \big\}, \\
W_0 & :=G^{-1}(]s_{\max},+\infty[). 
\end{align*}
It is not hard to see that $\frak{B}$ extends to a conformal isomorphism $W_0 \to \{ w : |w| > \e^{s_{\max}} \}$ which still satisfies the conjugacy relation \eqref{bfe} for $z \in W_0$. If $K_P$ is connected, every critical point of $P$ belongs to $K_P$, so $s_{\max}=0$. In this case $W_0=\CC \sm K_P$ and $\frak{B}$ is a conformal isomorphism $\CC \sm K_P \to \CC \sm \Dbar$. If $K_P$ is disconnected, there is at least one critical point of $P$ in $\CC \sm K_P$, so $s_{\max}>0$. In this case $W_0$ is a domain bounded by the piecewise analytic equipotential curve $G=s_{\max}$ containing the fastest escaping critical point(s) of $P$.

\subsection{Generalized rays}\label{gr}

For $\theta \in \TT$, we denote by $R_\theta$ the maximally extended smooth field line of $\nabla G$ such that $\frak{B}(W_0 \cap R_\theta)$ is the radial line $\{ \e^{s+2\pi i \theta}: s>s_{\max} \}$. We can parametrize $R_\theta$ by the potential, so for each $\theta$ there is an $s_\theta \geq 0$ such that $G(R_\theta(s)) = s$ for all $s > s_\theta$.\footnote{This amounts to viewing $R_\theta$ as a trajectory of the vector field $\nabla G / \| \nabla G \|^2$.} The field line $R_\theta$ either extends all the way to the Julia set $\bd K_P$ in which case $s_\theta=0$, or it crashes into a critical point $\om$ of $G$ at potential $s_\theta>0$ in the sense that $\lim_{s \to s_{\theta}^+} R_\theta(s)=\om$. \vs

A critical point $\om$ of $G$ of order $n$ is the starting point of $n$ ascending and $n$ descending field lines for $\nabla G$ that alternate around $\om$ (see \figref{cp} for the case $n=3$). Thus, at most $n$ field lines of the form $R_\theta$ can crash into $\omega$.  

\begin{lemma}\label{sts}
\mbox{}
\begin{enumerate}
\item[(i)]
The function $\theta \mapsto s_\theta$ is upper semicontinuous on $\TT$. \vs
\item[(ii)]
For every $s>0$, the set $\{ \theta \in \TT: s_\theta>s \}$ is finite. \vs 

\item[(iii)] 
For every $\theta \in \TT$, 
$$
s_{\whD(\theta)} \leq D s_\theta.
$$
Equality holds if $R_\theta$ does not crash into a critical point of $P$.
\end{enumerate}
\end{lemma}

\begin{proof}
(i) is the statement that if $s_{\theta_0}<s$ for some $\theta_0$, then $s_\theta<s$ for all $\theta$ close to $\theta_0$. This is a simple consequence of the fact that the trajectories of smooth vector fields depend continuously on their initial point. \vs

(ii) follows from the fact that for every $s>0$ there are finitely many critical points of $G$ at potentials higher than $s$, and each of them can have only finitely many field lines crashed into it. \vs

For (iii), first note that the image $P(R_\theta)$ is a smooth field line which by \eqref{bfe} maps under $\frak{B}$ to the radial line at angle $\whD(\theta)$. Thus 
\begin{equation}\label{kuku}
P(R_\theta(s))=R_{\whD(\theta)}(Ds) \qquad \text{for all} \ s>s_\theta.
\end{equation}
This proves $s_{\whD(\theta)} \leq D s_\theta$. Now suppose $R_\theta$ does not crash into a critical point of $P$. Then there are two possibilities: (i) $R_\theta$ does not crash at all, so $s_\theta=0$. In this case \eqref{kuku} shows that $R_{\whD(\theta)}$ does not crash either and $s_{\whD(\theta)}=0$; (ii) $R_\theta$ crashes into a strictly precritical point $\omega$ of $P$. In this case \eqref{kuku} shows that $R_{\whD(\theta)}$ crashes into $P(\omega)$ which is a critical point of $G$ at potential $D s_\theta$, proving once again $s_{\whD(\theta)}=D s_\theta$. 
\end{proof}

\begin{figure}[t]
\centering
\begin{overpic}[width=0.5\textwidth]{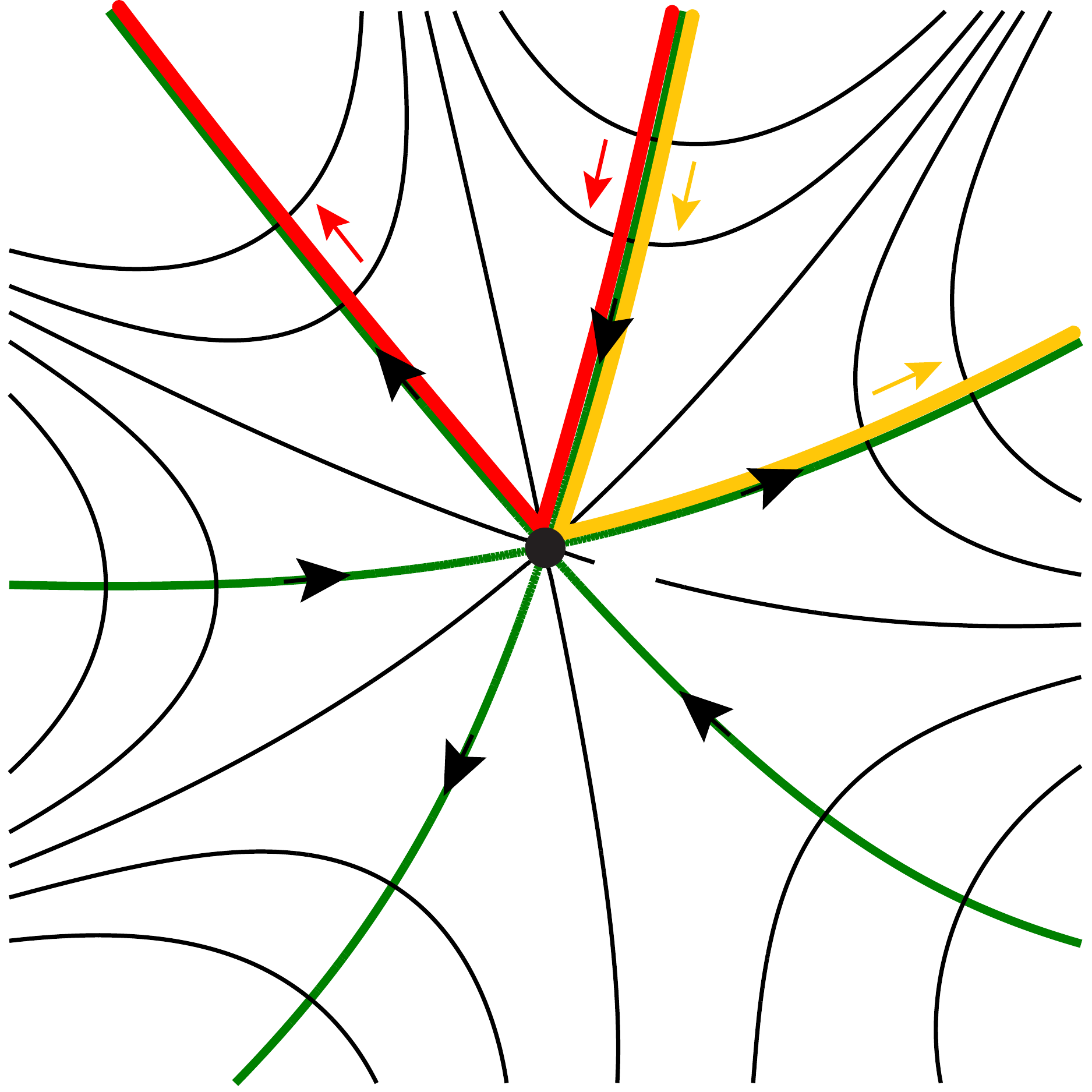}
\put (55.5,46.5) {$\omega$}
\put (64.5,93) {\small $R^-_{\theta}$}
\put (51,93.5) {\small $R^+_{\theta}$}
\end{overpic}
\caption{\sl Field lines and equipotentials of the Green's function near a critical point $\omega$ of order $3$. Each of the three incoming field lines can be extended past $\omega$ by turning to the immediate right or left and continuing along the corresponding outgoing line field.}  
\label{cp}
\end{figure}

It follows from the upper semicontinuity of $\theta \mapsto s_\theta$ that the set 
$$
\Sigma : = \bigcup_{\theta \in \TT} \bigcup_{s \in ]s_\theta,+\infty[} \{  \e^{s+2\pi i \theta} \}
$$
is open. It is not hard to see that the extension of the B\"{o}ttcher coordinate defined by $\frak{B}(R_\theta(s)):=\e^{s+2\pi i \theta}$ gives a conformal isomorphism $\frak{B}: W \to \Sigma$, where 
\begin{equation}\label{omgdef}
W : = \bigcup_{\theta \in \TT} \bigcup_{s \in ]s_\theta,+\infty[} \{ R_\theta(s) \}
\end{equation} 
Observe that $W$, being homeomorphic to the star-shaped domain $\Sigma$, is simply connected.  

\begin{corollary}
The set ${\mathcal N} \subset \TT$ of angles $\theta \in \TT$ for which $s_\theta>0$ is countable, dense and backward-invariant under $\whD$.
\end{corollary}

This follows from the fact that there are countably many critical points of $G$ in $\CC \sm K_P$, and that $s_\theta>0$ whenever $s_{\whD(\theta)}>0$ by \lemref{sts}. \vs

Here is a closely related description of the set $\mathcal N$:

\begin{corollary}\label{N0N}
Let ${\mathcal N}_0$ be the finite set of angles $\theta \in \TT$ for which the field line $R_\theta$ crashes into a critical point of $P$ in $\CC \sm K_P$. Then 
$$
{\mathcal N} = \bigcup_{n \geq 0} \whD^{-n}({\mathcal N}_0).
$$ 
\end{corollary}

\begin{proof}
The union is a subset of $\mathcal N$ since ${\mathcal N}_0 \subset {\mathcal N}$ and ${\mathcal N}$ is backward-invariant under $\whD$. Conversely, suppose $\theta \in {\mathcal N}$ so $R_\theta$ crashes into a critical point $\omega$ of $G$ in $\CC \sm K_P$. Let $n \geq 0$ be the smallest integer for which $c:=P^{\circ n}(\omega)$ is a critical point of $P$. An easy induction using \eqref{kuku} and the inequality $s_{\whD(\theta)} \leq D s_\theta$ gives 
$$
P^{\circ n}(R_\theta(s))=R_{\whD^{\circ n}(\theta)}(D^n s) \qquad \text{for} \ s>s_\theta.
$$
This implies that the field line $R_{\whD^{\circ n}(\theta)}$ crashes into $c$, which shows $\whD^{\circ n}(\theta) \in {\mathcal N}_0$. 
\end{proof}

When $\theta \notin {\mathcal N}$, the field line $R_\theta$ is called the {\bit smooth ray} at angle $\theta$. When $\theta \in {\mathcal N}$, the field line $R_\theta$ is defined only for $s > s_\theta$ and there is more than one way to extend it to a curve consisting of field lines and singularities of $\nabla G$ on which $G$ defines a homeomorphism onto $]0,+\infty[$. But there are always two special choices: $R_\theta^+$ which turns immediate right and $R_\theta^-$ which turns immediate left at each critical point met during the descent (compare \figref{cp}). We call these extensions the {\bit right} and {\bit left broken rays} at angle $\theta$, respectively. Note that these broken rays are also parametrized by the potential, so $R^{\pm}_{\theta}(s)$ make sense for all $s>0$, and
$$
R^+_{\theta}(s)=R^-_{\theta}(s)=R_\theta(s) \qquad \text{for} \ s > s_\theta. 
$$
On the other hand, an easy exercise shows that for each potential $s<s_\theta$ the points $R^{\pm}_\theta(s)$ belong to different connected components of $G^{-1}([0,s_\theta[)$, so the restrictions of the curves $s \mapsto R^\pm_\theta(s)$ to $]0,s_\theta[$ are disjoint. \vs
 
Each compact piece of a broken ray is the one-sided uniform limit of the corresponding piece of the nearby smooth rays in the following sense: Let $\theta_0 \in {\mathcal N}$ and fix $0<c<1$ such that $1/c>s_{\theta_0}>c>0$. By \lemref{sts} we have $s_\theta \leq c$ for all $\theta$ in a deleted neighborhood of $\theta_0$. Then,     
\begin{equation}\label{limit}
\lim_{\theta \nearrow \theta_0} R_{\theta}(s) = R_{\theta_0}^-(s) \quad \text{and} \quad 
\lim_{\theta \searrow \theta_0} R_{\theta}(s) = R_{\theta_0}^+(s)
\end{equation}
uniformly on $s \in [c,1/c]$. \vs

In what follows by a {\bit ray} we always mean a smooth or broken ray. It easily follows from \eqref{kuku} that
\begin{align*}
P(R_\theta) & = R_{\whD(\theta)} & & \text{if} \ \theta \notin {\mathcal N} \\
P(R_\theta^\pm) & = R_{\whD(\theta)}^\pm & & \text{if} \ \theta \in {\mathcal N} \ \text{and} \ \whD(\theta) \in {\mathcal N} \\
P(R_\theta^\pm) & = R_{\whD(\theta)} & & \text{if} \ \theta \in {\mathcal N} \ \text{and} \ \whD(\theta) \notin {\mathcal N}. 
\end{align*}
A critical point of $G$ of order $n$ belongs to $n$ left and $n$ right broken rays. On the other hand, a non-critical point in $\CC \sm K_P$ belongs either to a unique smooth ray or to a pair of left and right broken rays. \vs 

Finally, let us discuss the number of critical points of $G$ on a broken ray at angle $\theta \in {\mathcal N}$. We consider three cases: \vs

$\bullet$ {\it Case 1.} $\theta$ has infinite forward orbit under $\whD$. Then by \corref{N0N} there is a smallest integer $n \geq 1$ such that $\whD^{\circ n}(\theta) \notin {\mathcal N}$, so we have the orbit of distinct rays 
$$
R^\pm_{\theta} \stackrel{P}{\longrightarrow} R^\pm_{\whD(\theta)} \stackrel{P}{\longrightarrow} \cdots \stackrel{P}{\longrightarrow} R^\pm_{\whD^{\circ n-1}(\theta)} \stackrel{P}{\longrightarrow} R_{\whD^{\circ n}(\theta)}, 
$$
with the last ray being smooth. It follows that $R^\pm_{\theta}$ can contain only finitely many critical points of $G$. These critical points must eventually map (in at most $n-1$ iterations) to distinct critical points of $P$. In particular, $R^\pm_{\theta}$ can contain at most $D-2$ critical points of $G$. \vs

$\bullet$ {\it Case 2.} $\theta$ is periodic of period $q \geq 1$ under $\whD$. Then $R^{\pm}_\theta$ are mapped onto themselves under $P^{\circ q}$:
$$
P^{\circ q}(R^{\pm}_\theta(s))=R^{\pm}_\theta(D^q s) \qquad \text{for all} \ s>0. 
$$
Since $R^{\pm}_\theta$ first crash into a critical point of $G$ at potential $s_\theta$, they contain at least the critical points $R^{\pm}_\theta(s_\theta/D^{nq})$ for every $n \geq 0$.  \vs

$\bullet$ {\it Case 3.} $\theta$ is not periodic but there is a smallest integer $n \geq 1$ such that $\whD^{\circ n}(\theta)$ is periodic of period $q \geq 1$ under $\whD$. Combining the previous two cases, it follows that $R_\theta^\pm$ contain infinitely many critical points if $\whD^{\circ n}(\theta) \in {\mathcal N}$ and only finitely many critical points if $\whD^{\circ n}(\theta) \notin {\mathcal N}$. \vs

It follows from the above analysis that a periodic ray is either smooth or infinitely broken. In particular, {\it distinct periodic rays are always disjoint}.

\subsection{Polynomial-like maps}\label{plm}  

A holomorphic map $f: U_1 \to U_0$ between Jordan domains is {\bit polynomial-like} if $\ov{U_1} \subset U_0$ and if $f$ is proper. In this case $f$ has a well-defined mapping degree $d \geq 1$. In analogy with the polynomial case, the filled Julia set of $f$ is defined as the non-empty compact set $K_f = \{ z \in U_1 : f^{\circ n}(z) \in U_1 \ \text{for all} \ n \geq 0 \}$. According to Douady and Hubbard \cite{DH}, there is a polynomial $Q$ of degree $d$ and a quasiconformal homeomorphism $\varphi : U_0 \to \varphi(U_0)$ which satisfies 
$$
\varphi \circ f = Q \circ \varphi \qquad \text{in} \ U_1,
$$ 
with $\ov{\bd} \varphi=0$ almost everywhere on $K_f$. The relation $\varphi(K_f)=K_Q$ easily follows. We say that $f$ and $Q$ are {\bit hybrid equivalent} and that $\varphi$ is a {\bit hybrid conjugacy} between $f$ and $Q$. When $K_f$ is connected, the polynomial $Q$ is uniquely determined up to affine conjugacy. 

\section{Basic properties of the set $I_K$}\label{sec:bpi}

For the rest of the paper and unless otherwise stated, we fix a monic polynomial $P$ of degree $D \geq 3$. Our standing assumption is that the filled Julia set $K_P$ is disconnected and has a non-degenerate connected component $K$ which is periodic with period $k \geq 1$. This section is devoted to the construction of the set $I=I_K$ of the angles of external rays of $P$ that accumulate on $\bd K$ and establishing its basic properties. We would like to point out that much of the material related to the proof of \thmref{A} in this section and next can be presented in the language of ``external classes'' (see \cite{DH}). However, this would require roughly the same amount of effort as the more concrete approach taken here.    

\subsection{Construction of the set $I$}\label{conI} 

For $s>0$, consider the Jordan domain 
$$
V_s := \text{the connected component of} \ G^{-1}([0,s[) \ \text{containing} \ K.
$$
Then the $V_s$ are nested and $K= \bigcap_{s>0} V_s$. For sufficiently small $s>0$ the restriction {\mapfromto {P^{\circ k}|_{V_s}} {V_s} {V_{D^k s}}} is a polynomial-like map of some degree $2 \leq d < D$ (independent of $s$) with connected filled Julia set $K$. Let us denote by $s^\ast>0$ the largest $s$ for which this is true. Equivalently, $s^\ast$ can be characterized as the largest $s$ for which all critical points of $P^{\circ k}$ in $V_s$ belong to $K$.\vs

By \lemref{sts} there are at most finitely many $\theta \in {\mathcal N}$ for which the field line $R_\theta$ crashes at a potential higher than $s$. It follows that the set
$$
E_s := \{ \theta \in \TT : R_\theta \cap V_s \neq \es \}
$$
is a union of finitely many disjoint open intervals with endpoints belonging to ${\mathcal N}$. If $E_s \not= \TT$, the complement $\TT \sm E_s$ consists of equally many closed non-degenerate intervals. The closure 
$$
I_s := \overline{E}_s
$$
is thus a finite union of disjoint closed non-degenerate intervals with endpoints in ${\mathcal N}$. Evidently $I_s$ is the set of angles $\theta$ such that the field line $R_\theta$ or precisely one of the broken rays $R^{\pm}_\theta$ enters $V_s$. \vs 

It will be convenient to call each complementary component of a compact subset of $\TT$ a {\bit gap} of that set. Each gap of $I_s$ is an open interval of the form $]\theta_1, \theta_2[$ where the broken rays $R^-_{\theta_1}$ and $R^+_{\theta_2}$ crash into a critical point of $G$ at some potential $\geq s$ and enter $V_s$ along a common field line (see \figref{Es}). We call such $R^-_{\theta_1}, R^+_{\theta_2}$ a {\bit ray pair} in $I_s$. It is not hard to see that every ray pair in $I_s$ arises from a gap in this fashion, that is, if $\theta_1, \theta_2 \in I_s$ and if $R^-_{\theta_1}$ and $R^+_{\theta_2}$ crash into a critical point at a potential $\geq s$, then $]\theta_1, \theta_2[$ is a gap of $I_s$.     

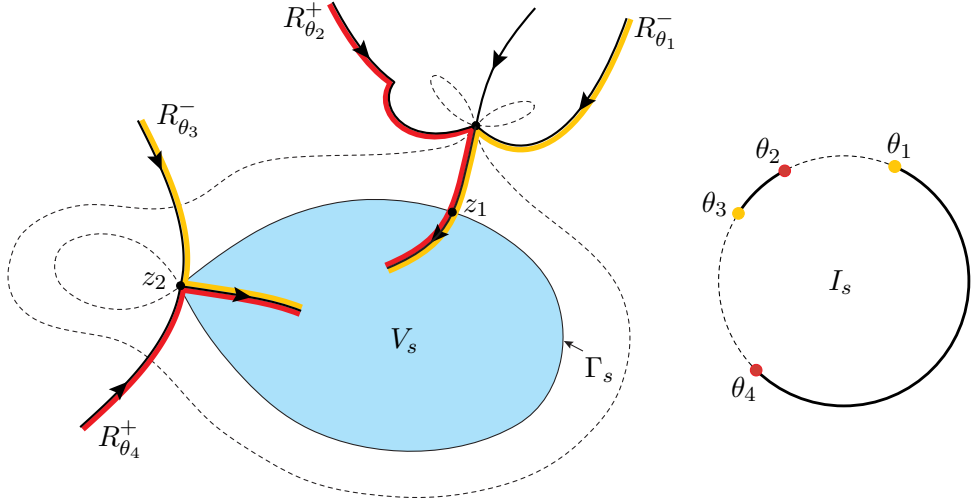
\begin{figure}[t]
\centering
\begin{overpic}[width=0.85\textwidth]{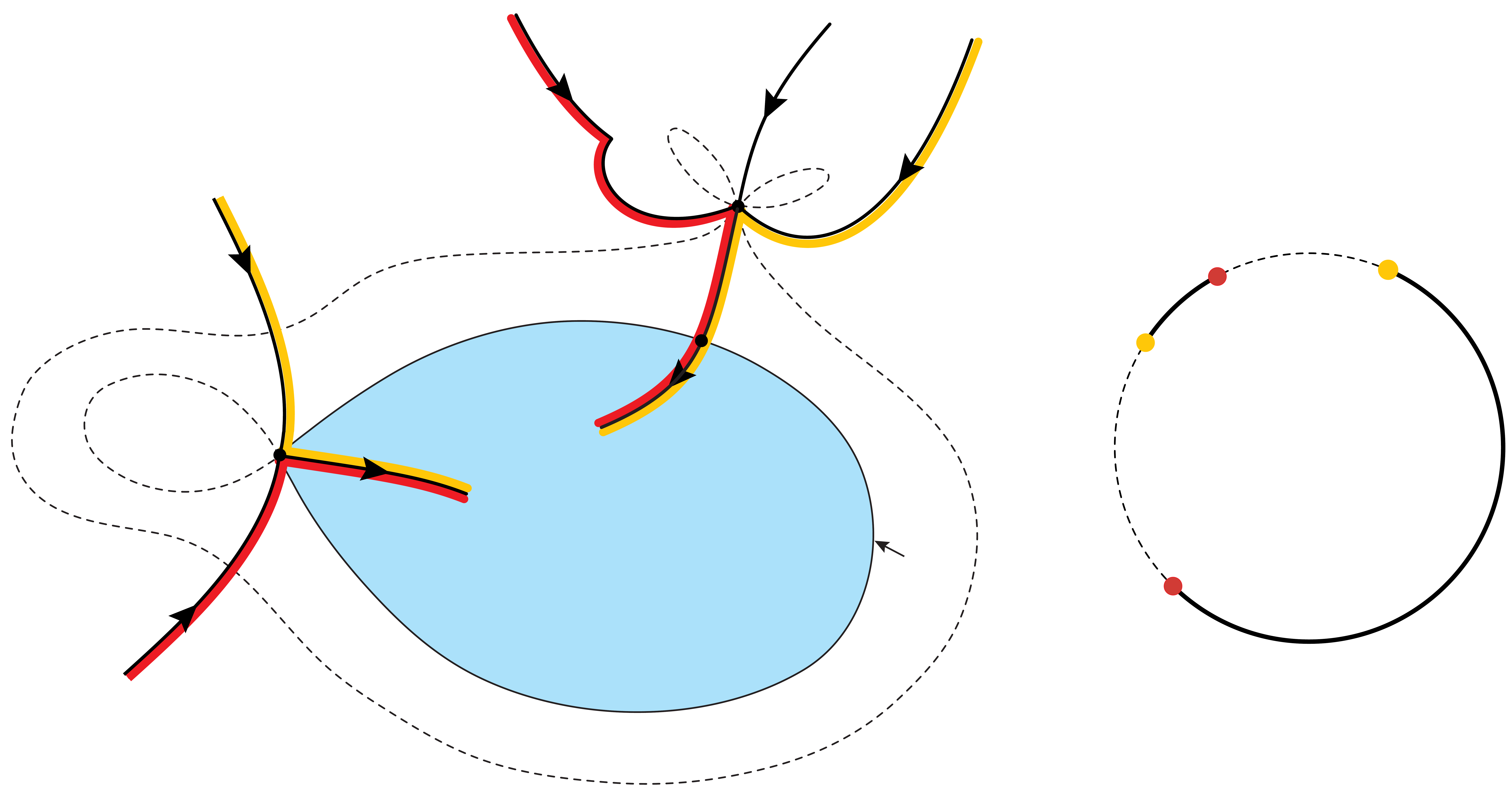}
\put (91,36) {\small $\theta_1$}
\put (77.5,35.5) {\small $\theta_2$}
\put (72,30) {\small $\theta_3$}
\put (75,11) {\small $\theta_4$}
\put (40,16) {\small $V_s$}
\put (60,13) {\small $\Gamma_s$}
\put (85,22) {\small $I_s$}
\put (65,48) {\small $R^-_{\theta_1}$}
\put (29,49) {\small $R^+_{\theta_2}$}
\put (16,39) {\small $R^-_{\theta_3}$}
\put (10,6) {\small $R^+_{\theta_4}$}
\put (47.5,30) {\small $z_1$}
\put (14.5,22.5) {\small $z_2$}
\end{overpic}
\caption{\sl The set $I_s$ is the closure of the set of angles of the field lines from $\infty$ that enter the topological disk $V_s$. Here $I_s$ has two gaps $]\theta_1, \theta_2[$ and $]\theta_3, \theta_4[$ corresponding to the ray pairs $R^-_{\theta_1}, R^+_{\theta_2}$ and $R^-_{\theta_3}, R^+_{\theta_4}$, and the topological boundary $\Ga_s=\bd V_s$ has two root points $z_1,z_2$ (see \S \ref{asep}).}  
\label{Es}
\end{figure}

\begin{lemma}\label{isk}
$\whD^{\circ k}(I_s)=I_{D^k s}$ whenever $0<s<s^\ast$.  
\end{lemma}

\begin{proof}
If $\theta \in E_s$, the field line $R_\theta$ enters $V_s$. Since $P^{\circ k}(V_s)=V_{D^k s}$, it follows that $R_{\whD^{\circ k}(\theta)} \supset P^{\circ k}(R_\theta)$ enters $V_{D^k s}$, so $\whD^{\circ k}(\theta) \in E_{D^k s}$. This proves $\whD^{\circ k}(E_s) \subset E_{D^k s}$ and the inclusion $\whD^{\circ k}(I_s) \subset I_{D^k s}$ follows. \vs

Now take any $\theta' \in E_{D^k s}$, let $\zeta$ be the intersection point of $R_{\theta'}$ with the boundary of $V_{D^k s}$ and find $z$ on the boundary of $V_s$ such that $P^{\circ k}(z)=\zeta$. If $z$ belongs to $R_\theta$ for some $\theta$, then $\theta \in E_s$ and $R_{\whD^{\circ k}(\theta)} \supset P^{\circ k}(R_\theta)$ passes through $\zeta$, so $\theta' = \whD^{\circ k}(\theta) \in \whD^{\circ k}(E_s)$. Otherwise, $z$ belongs to some broken ray $R^+_\theta$ which enters $V_s$. Then $\theta \in I_s$ and $P^{\circ k}(R^+_\theta) = R_{\whD^{\circ k}(\theta)}$ or $R^+_{\whD^{\circ k}(\theta)}$ passes through $\zeta$. Since $s_{\theta'}<D^k s$, we conclude again that $\theta' = \whD^{\circ k}(\theta) \in \whD^{\circ k}(I_s)$. This proves $E_{D^k s} \subset \whD^{\circ k}(I_s)$ and the reverse inclusion $I_{D^k s} \subset \whD^{\circ k}(I_s)$ follows. 
\end{proof}

\begin{remark}\label{excep}
Since $\whD$ is an open map, it follows from the above lemma that $\whD^{\circ k}(E_s) \subset E_{D^k s}$. On the other hand, a point on the boundary $\bd I_s$ may well map to an interior point in $E_{D^k s}$, although this should be thought of as a rare occurrence. In fact, if $\theta \in \bd I_s$ and $\whD^{\circ k}(\theta) \in E_{D^k s}$ for some $0<s<s^{\ast}$, then $\theta$ must belong to the finite set $\bigcup_{n=0}^{k-1} \whD^{-n}({\mathcal N}_0)$ (recall from \S \ref{gr} that ${\mathcal N}_0$ is the set of angles of rays which first crash into a critical point of $P$ in $\CC \sm K_P$). To see this, simply note that if $\theta, \whD(\theta), \ldots, \whD^{\circ k-1}(\theta)$ were all outside ${\mathcal N}_0$, repeated application of \lemref{sts} would give $s_{\whD^{\circ k}(\theta)}=D^k s_\theta \geq D^k s$, which would contradict $\whD^{\circ k}(\theta) \in E_{D^k s}$.   
\end{remark}

The sets $E_s$ and $I_s$ form nested families since if $s<s'$ and $R_\theta$ enters $V_s$, then it also enters $V_{s'}$. The intersection
$$
I := \bigcap_{s>0} I_s
$$
is thus a non-empty compact subset of $\TT$ which by \lemref{isk} satisfies $\whD^{\circ k}(I)=I$.  

\begin{lemma}
$I$ is the set of angles $\theta \in \TT$ for which the smooth ray $R_\theta$ accumulates on $\bd K$ if $\theta \notin {\mathcal N}$, or (precisely) one of the broken rays $R^{\pm}_\theta$ accumulates on $\bd K$ if $\theta \in {\mathcal N}$.
\end{lemma}

\PROOF
If $\theta \notin {\mathcal N}$ and $R_\theta$ accumulates on $\bd K$, then $R_\theta$ enters $V_s$ for every $s>0$, so $\theta \in \bigcap_{s>0} E_s \subset I$. If $\theta \in {\mathcal N}$ and, say, $R^+_{\theta}$ accumulates on $\bd K$ (the case of $R^-_{\theta}$ is similar), then by \eqref{limit} for every $s>0$ there is an $\ve>0$ such that $R_{\theta'}$ enters $V_s$ if $\theta<\theta'<\theta+\ve$. Any such $\theta'$ belongs to $E_s$, so $\theta \in I_s$. Since this holds for all $s>0$, it follows that $\theta \in I$. \vs

Conversely, suppose $\theta \in I$. If $\theta \notin {\mathcal N}$, then $\theta \in E_s$ for every $s>0$. It follows that the smooth ray $R_\theta$ enters every $V_s$, so it accumulates on $\bd K$. If $\theta \in {\mathcal N}$ and $0<s<s_\theta$, then $\theta$ is a boundary point of $I_s$, so precisely one of the broken rays $R^{\pm}_\theta$ enters $V_s$. Since this holds for every $0<s<s_\theta$, it follows that this broken ray accumulates on $\bd K$.    
\ENDPROOF 

\subsection{The affine structure of equipotential curves}\label{asep}

Recall that for $s>0$ the topological disk $V_s$ is the connected component of $G^{-1}([0,s[)$ containing $K$. Let us revisit the set $E_s$ of angles of the field lines $R_\theta$ that enter $V_s$, and the closure $I_s= \ov{E_s}$. For each gap $]\theta_1, \theta_2[$ in $I_s$, the broken rays $R^-_{\theta_1}$ and $R^+_{\theta_2}$ crash at some potential $\geq s$ and enter $V_s$ along a common field line. The intersection of this field line with the topological boundary 
$$
\Ga_s := \bd V_s
$$
is called a {\bit root} point of $\Ga_s$ (see \figref{Es}). \vs  

The extended B\"{o}ttcher coordinate $\frak{B}$ defined in the simply connected domain $W$ of \eqref{omgdef} can be used to define a canonical affine structure on the equipotential curves $\Ga_s$. Recall that $\frak{B}: W \to \Sigma$ is a conformal isomorphism which satisfies $\frak{B} \circ P=\frak{B}^D$ in $W$. The function 
$$
\Theta := \frac{1}{2\pi}\arg(\frak{B}) 
$$
is harmonic and well-defined up to an additive integer.\footnote{It is easy to check that $2\pi \Theta$ is a harmonic conjugate of $G$ in $W$.} We will think of $\Theta$ as a map $W \to \TT$. \vs 

The restriction of $\Theta$ gives a homeomorphism 
$\Ga_s \sm \{ \text{root points} \} \to E_s$. At a root point of $\Ga_s$ this map has a jump discontinuity. In fact, if $z_0$ is a root point corresponding to the gap $]\theta_1, \theta_2[$ in $I_s$, then in the positive orientation of $\Ga_s$ we have $\lim_{z \nearrow z_0} \Theta(z)= \theta_1$ and $\lim_{z \searrow z_0} \Theta(z)= \theta_2$. The inverse map $E_s \to \Ga_s \sm \{ \text{root points} \}$ thus extends continuously to a surjective map $h_s : I_s \to \Ga_s$ which is homeomorphic on each connected component of $I_s$ but identifies pairs of gap endpoints by sending them to the corresponding root. Similarly, consider  a piecewise affine surjection $\Pi_s : I_s \to \TT$ that has constant slope $1/|I_s|$ on each connected component of $I_s$ and identifies pairs of gap endpoints. Thus, there is an orientation-preserving homeomorphism $\psi_s: \Gamma_s \to \TT$ which makes the following diagram commute:
\begin{equation}\label{hhh}
\begin{tikzcd}[column sep=small]
I_s \arrow[drr,swap,"\Pi_s"] \arrow[rr,"h_s"] & & \Gamma_s \arrow[d,"\psi_s"] \\
 & & \TT
\end{tikzcd} 
\end{equation}
Evidently $\Pi_s$, and therefore $\psi_s$, is unique up to a rigid rotation of the circle $\TT$. Thus there is a unique affine structure on $\Ga_s$ with respect to which any choice of $\psi_s$ is an affine homeomorphism. This structure allows us to talk about affine maps between various $\Ga_s$'s. It also equips $\Ga_s$ with a well-defined metric coming from the Euclidean metric on $\TT$. In particular, the angular length of $\Ga_s$ is
$$
|\Ga_s| = |I_s|.
$$
In what follows we always measure lengths and slopes on $\Ga_s$ with respect to this affine structure. 

\begin{lemma}\label{slp}
For $0<s<s^\ast$ the following diagram is commutative: 
\begin{equation}\label{glue}
\begin{tikzcd}[column sep=small]
I_s \arrow[d,swap,"h_s"] \arrow[rr,"\whD^{\circ k}"] & & I_{D^k s} \arrow[d,"h_{D^k s}"] \\
\Ga_s \arrow[rr,"P^{\circ k}"] & & \Ga_{D^k s} 
\end{tikzcd} 
\end{equation}
In particular, the iterate $P^{\circ k}: \Ga_s \to \Ga_{D^k s}$ is an affine covering map of degree $d$ and slope $D^k$.
\end{lemma}

\begin{proof}
For simplicity set $s':=D^k s$. Both sets $\Ga_s \sm \{ \text{root points} \}$ and $\Ga_{s'} \sm \{ \text{root points} \}$ are contained in $W$ and $P^{\circ k}$ maps the former to the latter. Since $\frak{B} \circ P^{\circ k} = \frak{B}^{D^k}$ in $W$, we obtain 
$$
\Theta \circ P^{\circ k} = \whD^{\circ k} \circ \Theta \qquad \text{on} \ \Ga_s \sm \{ \text{root points} \}. 
$$ 
The result follows since $h_s$ and $h_{s'}$ are the inverses of the restrictions of $\Theta$ to $\Ga_s \sm \{ \text{root points} \}$ and $\Ga_{s'} \sm \{ \text{root points} \}$, respectively. 
\end{proof}

Consider the continuous retraction {\mapfromto {\rho_s} {\C \sm V_s} {\Ga_s}} defined by projecting along rays. More precisely, take $\zeta \in \C \sm V_s$ and consider two cases. If $\zeta \in R_\theta$ for some $\theta \in E_s$, let $\rho_s(\zeta)$ be the unique point in $\Ga_s \cap R_\theta$. Otherwise, $\zeta$ belongs to a ray whose angle is in the closure $[\theta_1,\theta_2]$ of a gap of $I_s$. In this case, let $\rho_s(\zeta)$ be the root point of $\Ga_s$ determined by $R_{\theta_1}^- , R_{\theta_2}^+$ (see \figref{Es}). Evidently, for any $t>s$ the restriction $\rho_s: \Ga_t \to \Ga_s$ is piecewise affine of degree $1$, with slope $1$ on the open arcs corresponding to $E_s$ and slope $0$ on the arcs corresponding to $E_t \sm I_s$. \vs

Now let $0<s<s^\ast$ so the restriction {\mapfromto {P^{\circ k}} {\Ga_s}{\Ga_{D^k s}}} is a covering map of degree $d$. Consider an affine homeomorphism {\mapfromto {\psi_s} {\Ga_s} \TT} with slope $1/|\Ga_s|$, as in \eqref{hhh}. By \lemref{slp} the composition 
$$
g_s := \psi_s \circ \rho_s \circ P^{\circ k} \circ \psi_s^{-1} : \TT \to \TT
$$ 
has degree $d$ and is piecewise affine with slopes $D^k$ and $0$. Let us call a fixed point $p=g_s(p)$ near which $g_s$ is not constant {\bit semi-repelling}. The terminology is justified by the fact that the derivative or a one-sided derivative of $g_s$ at such $p$ is $D^k>1$.   

\begin{lemma}\label{fpcount}
The map $g_s$ has at least $d-1$ semi-repelling fixed points on the circle.
\end{lemma}

(Compare \cite{PZ2} for a closely related result.) 

\begin{proof}
Take any lift $\tilde{g}_s: \RR \to \RR$ of $g_s$ and consider the function $T(t)=\tilde{g}_s(t)-t$ which is piecewise affine with slopes $D^k-1>0$ and $-1$ and satisfies $T(t+1)=T(t)+d-1$ for all $t$. The fixed points of $g_s$ correspond to the points in $[0,1[$ at which $T$ takes an integer value, and being semi-repelling means $T$ does not have slope $-1$ there. By the intermediate value theorem, for each of the $d-1$ integers $j$ satisfying $T(0) \leq j < T(1)=T(0)+d-1$, the equation $T(t)=j$ has at least one solution in $[0,1[$. Evidently $T$ cannot have slope $-1$ at all such solutions, so at least one of them must correspond to a semi-repelling fixed point of $g_s$. 
\end{proof} 

By the construction, each semi-repelling fixed point of $g_s$ corresponds to a fixed point of $\whD^{\circ k}$ in $I_s$. Since the $I_s$ are nested with $I = \bigcap_{s>0} I_s$, and since $\whD^{\circ k}$ has only finitely many fixed points on the circle, we conclude from \lemref{fpcount} the following

\COR \label{fps}
There exist at least $d-1$ points $\theta \in I$ such that $\whD^{\circ k}(\theta) = \theta$. 
\ENDCOR

The number of fixed points of $\whD^{\circ k}$ in $I$ could be greater than $d-1$; see the examples in \S \ref{sec:ex}. 

\begin{remark}
The periodic points of $\whD$ in the above corollary are all of minimal period $k$. In fact, if $\whD^{\circ m}(\theta)=\theta$ for some $\theta \in I$ and $m>0$, then $P^{\circ m}(R_\theta)=R_\theta$ or $P^{\circ m}(R^\pm_\theta)=R^\pm_\theta$ accumulates on $K$, which implies $P^{\circ m}(K) \cap K \neq 0$. As the component $K$ has minimal period $k$ under $P$, we conclude that $m$ must be a multiple of $k$.
\end{remark}

\section{Proof of Theorem \ref{A}}\label{sec:pfsa}

We now have all the ingredients for the proof of \thmref{A}. The existence of the semiconjugacy and the Hausdorff dimension bound will be proved in \S \ref{consp}. The uniqueness of $\Pi$ will be addressed in \S \ref{uniqq}. 
 
\subsection{Construction of the semiconjugacy $\Pi$}\label{consp}

Let $\theta_0$ be a fixed point of $\whD^{\circ k}$ in $I$ whose existence is guaranteed by \corref{fps}. Fix any $s_0$ with $0<s_0<s^\ast$ and consider the potentials 
$$
s_n := D^{-nk} s_0 \qquad \text{for} \ n \geq 0. 
$$
For simplicity, we write $I_n$ for $I_{s_n}$, $\Ga_n$ for $\Ga_{s_n}$, and so on. Let {\mapfromto {\Pi_n} {I_n} \TT} denote the unique piecewise affine surjection with slope $1/|I_n|$, normalized so that $\Pi_n(\theta_0) =0$, and consider the induced orientation-preserving affine homeomorphism {\mapfromto {\psi_n} {\Ga_n} \TT} which by \eqref{hhh} satisfies 
\begin{equation}\label{php}
\psi_n \circ h_n = \Pi_n.
\end{equation}
For $n>0$, the composition $\psi_{n-1} \circ P^{\circ k} \circ \psi_n^{-1} : \TT \to \TT$ is a degree $d$ covering map that fixes $0$ and has constant slope, so it must be the map $\whd$. Thus,
$$
\psi_{n-1} \circ P^{\circ k} = \whd \circ \psi_n \qquad \text{on} \ \Ga_n.
$$
Using \eqref{glue} and \eqref{php}, we obtain the relation 
$$
\Pi_{n-1} \circ \whD^{\circ k} = \whd \circ \Pi_n \qquad \text{on} \ I_n,
$$
which can be visualized as the infinite commutative diagram
\begin{equation}\label{projn}
\begin{tikzcd}[column sep=small]
\cdots \arrow[rr,"\whD^{\circ k}"] & & I_{n+1} \arrow[d,"\Pi_{n+1}"] \arrow[rr,"\whD^{\circ k}"] & & I_n \arrow[d,"\Pi_n"] \arrow[rr,"\whD^{\circ k}"] & & \arrow[rr,"\whD^{\circ k}"] \cdots & & I_1 \arrow[d,"\Pi_1"] \arrow[rr,"\whD^{\circ k}"] & & I_0 \arrow[d,"\Pi_0"] \\
\cdots \arrow[rr,"\whd"] & & \TT \arrow[rr,"\whd"] & & \TT \arrow[rr,"\whd"] & & \arrow[rr,"\whd"] \cdots & & \TT \arrow[rr,"\whd"] & & \TT 
\end{tikzcd} 
\end{equation}
Note that this implies 
$$
\frac{|I_{n+1}|}{|I_n|}=\frac{d}{D^k} \qquad \text{for all} \ n \geq 0, 
$$
which in particular shows $|I_n| \to 0$ and therefore $|I|=0$. More precisely, we have the following

\begin{theorem}
The Hausdorff dimension of the set $I=I_K$ is at most $\log d /(k \log D)$. 
\end{theorem}    

\begin{proof}
Let $B_n$ be the image of the set of boundary points of $I_n$ under the surjection $\Pi_n$. By \eqref{projn}, $\whd^{-1}(B_n) \subset B_{n+1}$. Moreover, by \remref{excep} any point in $B_{n+1} \sm \whd^{-1}(B_n)$ must belong to the finite set $\bigcup_{n=0}^{k-1} \whD^{-n}({\mathcal N}_0)$. This shows that the cardinality of $B_n$, i.e., the number of connected components of $I_n$, grows asymptotically as $d^n$. Evidently the connected components of $I_n$ have length $\leq \text{const.} \, D^{-nk}$. Thus the lower box dimension of $I$ is at most 
$$
\lim_{n \to \infty} \frac{\text{const.} + n \log d}{\text{const.} + nk \log D} = \frac{\log d}{k \log D}. 
$$
The result follows since the Hausdorff dimension is bounded above by the lower box dimension.   
\end{proof}

\begin{lemma}\label{dto1}
For every $\theta \in I$ there are $d$ distinct angles $\theta_1, \ldots, \theta_d$ in $\whD^{-k}(\theta) \cap I$ such that $\Pi_n$ is injective on $\{ \theta_1, \ldots, \theta_d \}$ for every $n$.
\end{lemma}
 
\begin{proof}
Take any $n \geq 0$. First suppose $\theta$ is an interior point of $I_n$. Then by \eqref{projn}, 
$$
\whD^{-k}(\theta) \cap I_{n+1} = \Pi_{n+1}^{-1} (\whd^{-1}(\Pi_n(\theta))).
$$
Under $\Pi_{n+1}$, every point in the $d$-element set $\whd^{-1}(\Pi_n(\theta))$ has either one preimage in the interior of $I_{n+1}$, or two preimages which are endpoints of a gap of $I_{n+1}$. By removing one of these endpoints in every such pair, we obtain a $d$-element set in $\whD^{-k}(\theta) \cap I_{n+1}$ on which $\Pi_{n+1}$ is injective. \vs

Now suppose $\theta$ is a right endpoint of a gap of $I_n$ (the case of a left endpoint is similar). Let $\theta'$ be the left endpoint of the same gap so $\Pi_n(\theta)=\Pi_n(\theta')$. By \lemref{isk}, all elements of $\whD^{-k}(\theta) \cap I_{n+1}$ are right endpoints of gaps in $I_{n+1}$ and all elements of $\whD^{-k}(\theta') \cap I_{n+1}$ are left endpoints of gaps in $I_{n+1}$. By \eqref{projn}, 
$$
(\whD^{-k}(\theta) \cup \whD^{-k}(\theta')) \cap I_{n+1} = \Pi_{n+1}^{-1} (\whd^{-1}(\Pi_n(\theta))).
$$   
This shows that under $\Pi_{n+1}$, every point in the $d$-element set $\whd^{-1}(\Pi_n(\theta))$ has two preimages, one in $\whD^{-k}(\theta) \cap I_{n+1}$ and the other in $\whD^{-k}(\theta') \cap I_{n+1}$. In particular,  $\whD^{-k}(\theta) \cap I_{n+1}$ consists of exactly $d$ elements and $\Pi_{n+1}$ is injective on it. \vs

The proof of the lemma is now straightforward. By what we just showed, for each $n$ there is a $d$-element set $X_n \subset \whD^{-k}(\theta) \cap I_n$ on which $\Pi_n$ is injective. Since there are only finitely many $d$-element subsets of $\whD^{-k}(\theta)$ on the circle, some set $X$ must occur infinitely often in the sequence $\{ X_n \}$. Evidently $X$ is contained in $\whD^{-k}(\theta) \cap I$ and $\Pi_n$ is injective on it for infinitely many $n$. To see that every $\Pi_n$ is injective on $X$, it suffices to observe that injectivity of $\Pi_{n+1}|_X$ implies injectivity of $\Pi_n|_X$. In fact, if $\theta_1, \theta_2$ are distinct points in $X$ with $\Pi_n(\theta_1)=\Pi_n(\theta_2)$, then $\theta_1, \theta_2$ are the endpoints of a gap of $I_n$, and since they both belong to $I$, they must be the endpoints of the same gap of $I_{n+1}$, which shows $\Pi_{n+1}(\theta_1)=\Pi_{n+1}(\theta_2)$.
\end{proof}

For $n \geq 0$ consider the set $Z_n:=\whd^{-n}(0)$ which consists of $d^n$ equally spaced rational points of the form $i/d^n \modd$. Clearly, $Z_0=\{ 0 \} \subset Z_1 \subset Z_2 \subset \cdots$. 

\begin{lemma}\label{Cn}
There is an increasing sequence of finite sets 
$$
C_0=\{ \theta_0 \} \subset C_1 \subset C_2 \subset \cdots \subset I
$$
such that for $n \geq 1$, \vs
\begin{enumerate}
\item[(i)]
$\Pi_j$ is injective on $C_n$ for every $j$; \vs
\item[(ii)]
$\whD^{\circ k}(C_n)=C_{n-1}$; \vs
\item[(iii)]
$\Pi_n(C_n)=Z_n$. 
\end{enumerate}  
\end{lemma}

Note that by (iii), $C_n$ has $d^n$ elements and $\Pi_n: C_n \to Z_n$ is an order-preserving bijection. 

\begin{proof}
We construct $\{ C_n \}$ inductively as follows. Set $C_0:= \{ \theta_0 \}$ and apply \lemref{dto1} to find a $d$-element set $C_1 \subset \whD^{-k}(\theta_0) \cap I$ containing $\theta_0$ on which every $\Pi_j$ is injective. Note that $\Pi_1(C_1) \subset Z_1$ by \eqref{projn}, hence $\Pi_1(C_1)=Z_1$ as both sets have $d$ elements. Suppose now that we have constructed the sets $C_0 \subset \cdots \subset C_m$ in $I$ which satisfy the conditions (i)-(iii) for all $1 \leq n \leq m$. For each $\theta \in C_m \sm C_{m-1}$ apply \lemref{dto1} to obtain a $d$-element set $X_\theta \subset \whD^{-k}(\theta) \cap I$ on which every $\Pi_j$ is injective. Define 
$$
C_{m+1} := C_m \cup \bigcup_{\theta \in C_m \sm C_{m-1}} X_\theta. 
$$     
Evidently $\whD^{\circ k}(C_{m+1})=C_m \subset C_{m+1}$. As the sets in the above union are disjoint, we see that $C_{m+1}$ has $d^{m+1}$ elements. By \eqref{projn} and the induction hypothesis, every $\Pi_j$ is injective on $C_{m+1}$, and $\Pi_{m+1}(C_{m+1}) \subset Z_{m+1}$. It follows that $\Pi_{m+1}(C_{m+1})=Z_{m+1}$, as both sets have $d^{m+1}$ elements. This completes the induction step. 
\end{proof}

As a consequence of the above lemma, we obtain the infinite commutative diagram
$$
\begin{tikzcd}[column sep=small]
\cdots \arrow[rr,"\whD^{\circ k}"] & & C_{n+1} \arrow[d,"\Pi_{n+1}"] \arrow[rr,"\whD^{\circ k}"] & & C_n \arrow[d,"\Pi_n"] \arrow[rr,"\whD^{\circ k}"] & & \arrow[rr,"\whD^{\circ k}"] \cdots & & C_1 \arrow[d,"\Pi_1"] \arrow[rr,"\whD^{\circ k}"] & & C_0 \arrow[d,"\Pi_0"] \\
\cdots \arrow[rr,"\whd"] & & Z_{n+1} \arrow[rr,"\whd"] & & Z_n \arrow[rr,"\whd"] & & \arrow[rr,"\whd"] \cdots & & Z_1 \arrow[rr,"\whd"] & & Z_0
\end{tikzcd} 
$$
in which every vertical arrow is an order-preserving bijection. Chasing around this commutative diagram shows that for all integers $j, \ell \geq 0$,
$$
\Pi_{j+\ell}(C_j)=Z_j.
$$
The proof is a straightforward induction on $j$ for each fixed $\ell$. It follows that for each $j \geq 0$,   
\begin{equation}\label{stable}
\Pi_n(\theta) = \Pi_j(\theta) \qquad \text{whenever} \ \theta \in C_j \ \text{and} \ n \geq j.
\end{equation}

It is now easy to construct the semiconjugacy $\Pi$ of \thmref{A}. Given $\theta \in I$, find adjacent points $x,y \in C_j$ such that $\theta \in [x,y]$. Then \eqref{stable} and the monotonicity of the projections show that if $n>m \geq j$, 
$$
\Pi_n(\theta)-\Pi_m(\theta) \leq \Pi_n(y) - \Pi_m(x) = \Pi_j(y)-\Pi_j(x) = \frac{1}{d^j}
$$
and
$$
\Pi_n(\theta)-\Pi_m(\theta) \geq \Pi_n(x) - \Pi_m(y) = \Pi_j(x)-\Pi_j(y) = -\frac{1}{d^j}.
$$
Since $1/d^j \to 0$ as $j \to \infty$, we conclude that the sequence $\{ \Pi_n \}$ converges uniformly on the compact set $I$ to a degree $1$ monotone surjection {\mapfromto \Pi I \TT} which by \eqref{projn} semiconjugates {\mapfromto {\whD^{\circ k}} I I} to {\mapfromto \whd \TT \TT}. 

\subsection{Uniqueness of $\Pi$}\label{uniqq}

To finish the proof of \thmref{A}, it remains to show that the semiconjugacy $\Pi$ constructed above is unique up to postcomposition with a rotation $\tau \mapsto \tau + j/(d-1) \modd$. We will prove this by first showing that the degree $1$ monotone extension of $\Pi$ semiconjugates a degree $d$ monotone extension of $\whD^{\circ k}|_I$ to $\whd$ (\lemref{sem}), and then invoking the well-known fact that such global semiconjugacies are unique up to a rotation (\corref{juju}).  \vs

Let us call a gap $J$ of the compact set $I$ {\bit minor} if $|J| < 1/D^k$ and {\bit major} if $|J| \geq 1/D^k$. The distinction depends on whether or not $\whD^{\circ k}$ acts homeomorphically on the closure of $J$. The {\bit multiplicity} of $J$ is the integer part of $D^k |J|$, that is, the number of times $\whD^{\circ k}$ wraps $J$ fully around the circle. A gap $J$ is {\bit taut} if $D^k |J|$ is an integer and {\bit loose} otherwise. Thus, minor gaps are always loose and have multiplicity $0$.  

\begin{lemma}\label{gapmap}
Suppose $]a,b[$ is a gap of $I$. Then either $\whD^{\circ k}(a)=\whD^{\circ k}(b)$ or $]\whD^{\circ k}(a),\whD^{\circ k}(b)[$ is a gap of $I$.  
\end{lemma}

\begin{proof}
We first prove a version of the claim for positive potentials: Suppose $]\theta_1, \theta_2[$ is a gap of $I_s$ for some $0<s<s^{\ast}$.
Set $s':=D^k s$ and consider $\theta'_i:=\whD^{\circ k}(\theta_i) \in I_{s'}$ for $i=1,2$. The ray pair $R^-_{\theta_1}, R^+_{\theta_2}$ crash into a critical point $\omega$ with potential $G(\omega) \geq s$, so the image rays $P^{\circ k}(R^-_{\theta_1}), P^{\circ k}(R^+_{\theta_2})$ have a common point $\omega':=P^{\circ k}(\omega)$. These rays are necessarily broken if $\theta'_1 \neq \theta'_2$, since two distinct smooth rays, or a smooth and a broken ray, can never meet. Thus $P^{\circ k}(R^-_{\theta_1})=R^-_{\theta'_1}$ and $P^{\circ k}(R^+_{\theta_2})=R^+_{\theta'_2}$ must crash into a critical point with potential $\geq G(\omega') = D^k G(\omega) \geq s'$. This shows that $R^-_{\theta'_1}, R^+_{\theta'_2}$ is a ray pair in $I_{s'}$, or equivalently $]\theta'_1, \theta'_2[$ is a gap of $I_{s'}$. \vs

Now consider a gap $]a,b[$ of $I$ such that $a':=\whD^{\circ k}(a)$ and $b':=\whD^{\circ k}(b)$ are distinct. Working with the sequence of potentials $s_n=D^{-nk}s_0$ as before, there is an integer $n_1 \geq 0$ and an increasing sequence $ \{ ]a_n,b_n[ \}_{n \geq n_1}$ such that $]a_n,b_n[$ is a gap of $I_n$ and $\bigcup_{n \geq n_1} ]a_n,b_n[ = ]a,b[$. We may assume that $a'_n:=\whD^{\circ k}(a_n)$ and $b'_n:=\whD^{\circ k}(b_n)$ are distinct. Then, by the positive potential case treated above, $]a'_n,b'_n[$ is a gap of $I_{n-1}$. Since there is an integer $n_2 \geq n_1$ such that the sequence $\{ ]a'_n,b'_n[ \}_{n \geq n_2}$ is increasing and $\bigcup_{n \geq n_2} ]a'_n,b'_n[ = ]a',b'[$, we conclude that $]a',b'[ \, \cap I = \es$. Now $a',b' \in I$ implies that $]a',b'[$ is a gap of $I$. 
\end{proof}

We can extend the restriction $\whD^{\circ k}|_I$ to a continuous monotone map $f: \TT \to \TT$ of {\it minimal} degree by sending each gap $]a,b[$ of $I$ homeomorphically onto the gap $]\whD^{\circ k}(a),\whD^{\circ k}(b)[$ if $\whD^{\circ k}(a) \neq \whD^{\circ k}(b)$, and to the point $\whD^{\circ k}(a)=\whD^{\circ k}(b)$ otherwise. Thus, the image $f(J)$ is a gap or a single point in $I$ according as $J$ is a loose or taut gap.  

\begin{lemma}\label{sem}
Let $\xi: I \to \TT$ be any semiconjugacy between $\whD^{\circ k}|_I$ and $\whd$. Then the degree $1$ monotone extension $\xi: \TT \to \TT$ is a semiconjugacy between $f$ and $\whd$:
$$
\begin{tikzcd}[column sep=small]
\TT \arrow[d,swap,"\xi"] \arrow[rr,"f"] & & \TT \arrow[d,"\xi"] \\
\TT \arrow[rr,"\whd"] & & \TT 
\end{tikzcd} 
$$
In particular, $f$ is a monotone map of degree $d$. 
\end{lemma} 

\begin{proof}
Evidently $\xi$ is constant on each gap of $I$. If $J=]a,b[$ is a taut gap, then $f=\whD^{\circ k}(a)$ in $J$, so $\xi \circ f = \xi(\whD^{\circ k}(a)) = \whd(\xi(a))= \whd \circ \xi$ in $J$. If $J$ is a loose gap, then $f(J)=]\whD^{\circ k}(a),\whD^{\circ k}(b)[$ is a gap by \lemref{gapmap}, so $\xi$ takes the constant value $\xi(\whD^{\circ k}(a))$ on it, and the relation $\xi \circ f = \whd \circ \xi$ in $J$ follows similarly.    
\end{proof}

\begin{corollary}\label{gapcount}
There are precisely $D^k-d$ major gaps in $I$ counting multiplicities. 
\end{corollary}

Compare \cite{BBM} and \cite{Z} for the similar case of ``rotation sets'' where $d=1$. 
 
\begin{proof}
Let $\{ J_i \}$ denote the countable collection of gaps of $I$ of multiplicities $\{ m_i \}$. We have $\sum_i |J_i| = 1$ since $I$ has measure zero, hence $\sum_i |f(J_i)| =d$ since $f$ has degree $d$. The definition of multiplicity shows that $|f(J_i)|=D^k |J_i| - m_i$ for each $i$. It follows that $D^k \sum_i |J_i| - \sum_i m_i =d$, or $\sum_i m_i = D^k-d$, as required. 
\end{proof}

\begin{corollary}\label{juju}
If $\xi_1, \xi_2 : I \to \TT$ are semiconjugacies between $\whD^{\circ k}|_I$ and $\whd$, there is an integer $j$ such that $\xi_1 = \xi_2 + j/(d-1) \modd$. 
\end{corollary}

\begin{proof}
Consider a fixed point $\theta_0= \whD^{\circ k}(\theta_0) \in I$ as in \S \ref{consp}. The images $\xi_i(\theta_0)$ belong to the set 
$$
\Big\{ \frac{0}{d-1}, \frac{1}{d-1}, \ldots, \frac{d-2}{d-1} \Big\} \modd 
$$
of fixed points of $\whd$. Since the rotation $\tau \mapsto \tau + 1/(d-1) \modd$ commutes with $\whd$, it suffices to show that if $\xi_1(\theta_0)=\xi_2(\theta_0)=0$, then $\xi_1=\xi_2$ everywhere. By \lemref{sem}, the degree $1$ monotone extensions $\xi_i: \TT \to \TT$ are semiconjugacies between $f$ and $\whd$. This will easily imply $\xi_1=\xi_2$ as follows (compare \cite{KH}). Let $x_0 \in \RR$ be a representative of $\theta_0 \in \TT$ and $F: \RR \to \RR$ be the unique lift of $f$ such that $F(x_0)=x_0$. If $\Xi_i: \RR \to \RR$ is the unique lift of $\xi_i$ with $\Xi_i(x_0)=0$, then $\Xi_i \circ F = d\, \Xi_i$. This means the $\Xi_i$ are the fixed points of the map $\Xi \mapsto (\Xi \circ F)/d$ acting on the complete metric space of continuous functions $\RR \to \RR$ that commute with $x \mapsto x+1$, equipped with the uniform metric $\d(\Xi,\Xi')=\sup_{x \in \RR} |\Xi(x)-\Xi'(x)|$. This map is clearly contracting by a factor $1/d<1$, hence it has a unique fixed point. We conclude that $\Xi_1=\Xi_2$ or $\xi_1=\xi_2$.           
\end{proof}

\section{Proof of Theorem \ref{B}}\label{sec:pfsb}

Throughout this section we will adopt the following notations: \vs

\begin{enumerate}[leftmargin=*]
\item[$\bullet$]
For $\theta \in I$, we denote by $R^P_\theta$ the unique external ray of $P$ at angle $\theta$ that accumulates on $\bd K$. Thus, $R^P_\theta$ is the smooth ray $R_\theta$ if $\theta \in I \sm {\mathcal N}$, and it is one of the broken rays $R^{\pm}_{\theta}$ if $\theta \in I \cap {\mathcal N}$. \vs

\item[$\bullet$]
$L_\theta$ is the radial line in $\CC \sm \ov{\DD}$ at angle $\theta \in \TT$: 
$$
L_\theta : = \{ r \e^{2 \pi i \theta} : r>1 \}. 
$$

\item[$\bullet$]
$Q_d: \CC \to \CC$ is the $d$-th power map $z \mapsto z^d$. \vs

\item[$\bullet$]
$\d_X$ is the distance in the hyperbolic metric of a domain $X \subset \Chat$ whose complement has at least three points. \vs 
\end{enumerate}

\subsection{A reduction}

We begin by reducing \thmref{B} to a statement on the hyperbolic geometry of rays. Suppose $\varphi: U_0 \to \varphi(U_0)$ is a hybrid conjugacy between $P^{\circ k}:U_1 \to U_0$ and a degree $d$ polynomial $Q$ which we may assume to be monic. In order to prove \thmref{B}, it suffices to show that there is a choice of the semiconjugacy $\Pi$ such that for every $\theta \in I$ the arc $\varphi(R^P_{\theta} \cap U_1)$ and the ray segment $R^Q_{\Pi(\theta)} \cap \varphi(U_1)$ have finite Hausdorff distance in the hyperbolic metric of $\varphi(U_0) \sm K_Q$.\footnote{Recall that the Hausdorff distance between two closed sets in a metric space is the infimum of the set of $\delta>0$ such that each set is contained in the $\delta$-neighborhood of the other. If there is no such $\delta$ the Hausdorff distance is defined to be $+\infty$.} This is because a ball of fixed radius in this metric has shrinking Euclidean diameter as its center tends to $\bd K_Q$. Use the B\"{o}ttcher coordinate $\frak{B}_Q: \CC \sm K_Q \to \CC \sm \ov{\DD}$ to form the composition $\Phi:=\frak{B}_Q \circ \varphi : U_0 \sm K \to \Phi(U_0 \sm K)$ which is a quasiconformal conjugacy between $P^{\circ k}: U_1 \sm K \to U_0 \sm K$ and $Q_d: \Phi(U_1 \sm K) \to \Phi(U_0 \sm K)$. Then the above condition is equivalent to $\Phi(R^P_{\theta} \cap U_1)$ and the radial segment $L_{\Pi(\theta)} \cap \Phi(U_1 \sm K)$ having finite Hausdorff distance in the hyperbolic metric of $\Phi(U_0 \sm K)$. Thus, \thmref{B} will follow from the following   

\begin{theorem}\label{B'}
Let $U_0, U_1$ (resp. $U'_0, U'_1$) be Jordan domains containing $K$ (resp. $\ov{\DD}$) such that $\ov{U_1} \subset U_0$ (resp. $\ov{U'_1} \subset U'_0$). Set $\Omega_i:=U_i \sm K$ and $\Om'_i:=U'_i \sm \ov{\DD}$ for $i=0,1$. Suppose $\Phi: \Om_0 \to \Om'_0$ is a quasiconformal conjugacy between $P^{\circ k} : \Om_1 \to \Om_0$ and the $d$-th power map $Q_d: \Om'_1 \to \Om'_0$. Then there is a choice of the semiconjugacy $\Pi: I \to \TT$ of \thmref{A} such that for each $\theta \in I$ the arc $\Phi(R^P_{\theta} \cap \Om_1)$ and the radial segment $L_{\Pi(\theta)} \cap \Om'_1$ have finite Hausdorff distance in the hyperbolic metric of $\Om'_0$.
\end{theorem}

In particular, the arc $\Phi(R^P_{\theta} \cap \Om_1)$ lands at the point $\exp(2 \pi i \Pi(\theta))$ on the unit circle. \vs 

The following lemma shows that it suffices to prove \thmref{B'} for {\it some} quasiconformal conjugacy: 

\begin{lemma}\label{immat}
If \thmref{B'} holds for one quasiconformal conjugacy between the restriction of $P^{\circ k}$ and $Q_d$, then it holds for any other. 
\end{lemma}

\begin{proof}
Suppose \thmref{B'} holds for $\Phi_1: \Om_0 \to \Om'_0$ and a corresponding semiconjugacy $\Pi_1: I \to \TT$. Let $\Phi_2$ be another such quasiconformal conjugacy which we may assume has the same domain $\Om_0$. The composition  
$$
\Psi := \Phi_2 \circ \Phi_1^{-1}: \Om'_0 \to \Om''_0:=\Phi_2(\Om_0)
$$
is a quasiconformal homeomorphism with $|\Psi(z)| \to 1$ as $|z| \to 1$, and therefore it extends to a homeomorphism of the unit circle. Moreover, since $\Psi$ commutes with $Q_d$, its extension to the unit circle is a rational rotation of the form $S: z \mapsto \e^{2\pi i j/(d-1)} z$. First consider the case where $S$ is the identity map. If $A$ is an annulus of the form $\{ z: 1<|z|<r \}$ with $r$ sufficiently close to $1$, it follows that 
\begin{equation}\label{disM}
\d_{\Om''_0}(\Psi(z),z) \leq M \qquad \text{for all} \ z \in A, 
\end{equation}
where $M>0$ depends only on the maximal dilatation of $\Psi$. By the assumption, for every $\theta \in I$ the arcs $\Phi_1(R^P_{\theta}) \cap A$ and $L_{\Pi_1(\theta)} \cap A$ have finite Hausdorff distance in $\Om'_0$. It follows from \eqref{disM} that $\Phi_2(R^P_{\theta}) \cap A$ and $L_{\Pi_1(\theta)} \cap A$ have finite Hausdorff distance in $\Om''_0$. Thus, \thmref{B'} holds for $\Phi_2$ and the choice $\Pi_1$. \vs

If $S$ is not the identity map, we can run the above argument for $S^{-1} \circ \Phi_2$ to conclude that \thmref{B'} holds for $\Phi_2$ and the choice $\Pi_2=\Pi_1+j/(d-1)$.   
\end{proof}
 
\subsection{\thmref{B'} and its corollaries}

The proof of \thmref{B'} begins as follows. Take the polynomial-like restriction $P^{\circ k}: U_1 \to U_0$, where $U_1=V_{D^{-k}s_0}$ and $U_0=V_{s_0}$ for a sufficiently small $s_0>0$ in the notation of \S \ref{conI}. We may choose the quasiconformal conjugacy $\Phi$ such that the images $\Om'_i=\Phi(\Om_i)$ for $i=0,1$ are round annuli of the form $\Om'_0=\{ z: 0< \log |z|< r_0 \}$ and $\Om'_1=\{ z: 0< \log |z|< d^{-1}r_0 \}$ for some $r_0>0$. Set 
\begin{align*}
\Om_n & := V_{D^{-nk}s_0} \sm K \\
\Om'_n & := \Phi(\Om_n)=Q_d^{-n}(\Om'_0)=\{ z: 0< \log |z| < d^{-n}r_0 \}.
\end{align*}
The topological annuli $\{ \Om_n \}$ and the round annuli $\{ \Om'_n \}$ are nested and the difference sets
$$
A_n:= \ov{\Om}_n \sm \Om_{n+1} \qquad \text{and} \qquad A'_n:= \ov{\Om'}_n \sm \Om'_{n+1}
$$
are closed annuli. The maps $P^{\circ k}: A_{n+1} \to A_n$ and $Q_d: A'_{n+1} \to A'_n$ are degree $d$ regular coverings for every $n \geq 0$. \vs   

If $\theta_0 \in I$ is a fixed point of $\whD^{\circ k}$ given by \corref{fps}, we may choose $\Phi$ such that $\Phi(R_{\theta_0}^P \cap A_0) = L_0 \cap A'_0$. Since $\Phi$ conjugates $P^{\circ k}: \Om_1 \to \Om_0$ to $Q_d: \Om'_1 \to \Om'_0$, it follows inductively that $\Phi (R_{\theta_0}^P \cap A_n) = L_0 \cap A'_n$ for all $n \geq 0$ and therefore 
\begin{equation}\label{r00}
\Phi(R_{\theta_0}^P \cap \Om_0) = L_0 \cap \Om'_0.
\end{equation}
By \thmref{A} there is a unique semiconjugacy $\Pi: I \to \TT$ between $\whD^{\circ k}|_I$ and $\whd$ normalized so that $\Pi(\theta_0)=0$. Let $\{ C_n \}$ be the sequence of iterated preimages of $\theta_0$ in $I$ constructed in \lemref{Cn}. 
 
\begin{lemma}\label{d-adic}
For every $n \geq 0$ and every $\theta \in C_n$,
$$
\Phi(R_\theta^P \cap \Om_n) = L_{\Pi(\theta)} \cap \Om'_n.
$$
\end{lemma}

\begin{proof}
Since $\whD^{\circ nk}(\theta)=\theta_0$, the ray segment $R_\theta^P \cap \Om_n$ maps under $P^{\circ nk}$ to $R_{\theta_0}^P \cap \Om_0$. It follows from \eqref{r00} that the arc $\Phi(R_\theta^P \cap \Om_n)$ maps under $Q_d^{\circ n}$ to $L_0 \cap \Om'_0$. Thus, $\Phi(R_\theta^P \cap \Om_n)=L_{\tau} \cap \Om'_n$ for some $\tau \in Z_n = \whd^{-n}(0)$. Evidently the assignment $\theta \mapsto \tau$ is an order-preserving bijection $C_n \to Z_n$ which by \eqref{r00} sends $\theta_0$ to $0$. Thus, by \lemref{Cn}, $\tau=\Pi_n(\theta)$. Finally, since $\theta \in C_n$, \eqref{stable} shows that $\Pi_n(\theta)=\Pi(\theta)$. 
\end{proof}

Now consider the open topological disks 
$$
\Delta := \mathring{A}_0 \sm R_{\theta_0}^P \qquad \text{and} \qquad  \Delta' := \Phi(\Delta) = \mathring{A}'_0 \sm L_0. 
$$
Observe that there are $d^n$ connected components of $Q_d^{-n}(\Delta')$ in $A'_n$ which are separated by the radial lines $L_\tau$ for $\tau \in Z_n$. Similarly, there are $d^n$ connected components of $P^{-nk}(\Delta)$ in $A_n$ which, in view of \lemref{d-adic}, are separated by the rays $R_\theta^P$ for $\theta \in C_n$.    

\begin{lemma}\label{samepiece}
For every $\theta \in I$ and every $n \geq 0$ there is a connected component of $Q_d^{-n}(\Delta')$ whose closure contains both arcs 
$$
\Phi(R_\theta^P \cap A_n) \qquad \textrm{and} \qquad L_{\Pi(\theta)} \cap A'_n.
$$ 
\end{lemma}

\begin{proof}
Take adjacent angles $x,y \in C_n$ such that $\theta \in [x,y[$; then $\Pi(\theta) \in [\Pi(x),\Pi(y)]$ by monotonicity. Let $\Delta_n$ be the unique connected component of $P^{-nk}(\Delta)$ whose closure contains both $R_x^P \cap A_n$ and $R_y^P \cap A_n$, and therefore $R_\theta^P \cap A_n$. It follows that the image $\Delta'_n=\Phi(\Delta_n)$ is a connected component of $Q_d^{-n}(\Delta')$ whose closure contains $\Phi(R_{\theta}^P \cap A_n)$. By \lemref{d-adic}, the closure of $\Delta'_n$ contains both $L_{\Pi(x)} \cap A'_n$ and $L_{\Pi(y)} \cap A'_n$, and therefore $L_{\Pi(\theta)} \cap A'_n$.  
\end{proof}

It is now easy to finish the proof of \thmref{B'}. 
For $n \geq 0$, let $\rho_n$ denote the hyperbolic metric of the annulus $\Om'_n$, so $\rho_0 < \rho_n$ in $\Om'_n$ by the Schwarz lemma. Choose $\delta>0$ so that the $d$ connected components of $Q_d^{-1}(\Delta')$ have diameter $< \delta$ in the metric $\rho_0$. Then any component of $Q_d^{-n}(\Delta')$ for $n \geq 2$ also has diameter $<\delta$ in the metric $\rho_0$ since the regular covering $Q_d^{\circ n}: (\Om'_n, \rho_n) \to (\Om'_0, \rho_0)$ is a local isometry and therefore $Q_d^{\circ n}: (\Om'_n, \rho_0) \to (\Om'_0, \rho_0)$ is expanding. It follows from \lemref{samepiece} that for each $\theta \in I$ the arcs $\Phi(R_\theta^P \cap \Om_1)$ and $L_{\Pi(\theta)} \cap \Om'_1$ are contained in the $\delta$-neighborhood of each other in the metric $\rho_0$, and therefore have Hausdorff distance $\leq \delta$ in $\Om'_0$. \qed 

\begin{corollary}\label{samefiber}
The following conditions on $\theta, \theta' \in I$ are equivalent: \vs
\begin{enumerate}
\item[(i)]
$\Pi(\theta)=\Pi(\theta')$. \vs
\item[(ii)]
$\d_{\CC \sm K}(R^P_\theta(s), R^P_{\theta'}(s))$ stays bounded as the Green's potential $s$ tends to $0$. \vs
\item[(iii)]
Under the (unique up to rotation) conformal isomorphism $\zeta: \CC \sm K \to \CC \sm  \ov{\DD}$, the arcs $\zeta(R^P_\theta)$ and $\zeta(R^P_{\theta'})$ land at the same point of the unit circle. \vs
\end{enumerate}
If these conditions are satisfied and $R^P_\theta$ lands at $z \in \bd K$, then $R^P_{\theta'}$ also lands at $z$. Moreover, one of the two connected components of $\CC \sm (R^P_{\theta} \cup R^P_{\theta'} \cup \{ z \})$ will be disjoint from $K$. 
\end{corollary}

\begin{proof}
We will use a quasiconformal conjugacy $\Phi$ and the associated objects, using the notations in the above proof of \thmref{B'}. \vs

(i) $\Longrightarrow$ (ii): By \lemref{samepiece}, for every $n \geq 0$ there is a connected component of $Q_d^{-n}(\Delta')$ whose closure contains both $\Phi(R^P_\theta \cap A_n)$ and $\Phi(R^P_{\theta'} \cap A_n)$. It follows that there is a connected component of $P^{-kn}(\Delta)$ in $A_n$ whose closure contains $R^P_\theta \cap A_n$ and $R^P_{\theta'} \cap A_n$. By a similar application of the Schwarz lemma as above, for $n \geq 1$ these components have uniformly bounded diameters in the hyperbolic metric of $\Om_0$. Thus, there is a $\delta>0$ such that for all $n \geq 1$,
$$
\d_{\Om_0}(R^P_\theta(s),R^P_{\theta'}(s)) < \delta \quad \text{if} \quad  D^{-(n+1)k}s_0 \leq s \leq D^{-nk}s_0.
$$ 
This proves $\d_{\Om_0}(R^P_\theta(s),R^P_{\theta'}(s)) < \delta$ for all $0<s \leq D^{-k}s_0$, and (ii) follows since the hyperbolic metrics of $\CC \sm K$ and $\Om_0$ are comparable in $\Om_1$. \vs

(ii) $\Longrightarrow$ (iii): Since the conformal isomorphism $\zeta$ is a hyperbolic isometry, $\d_{\CC \sm \ov{\DD}}(\zeta(R^P_\theta(s)),\zeta(R^P_{\theta'}(s)))$ stays bounded as $s \to 0$. In particular, $\zeta(R^P_\theta)$ and $\zeta(R^P_{\theta'})$ have the same accumulation sets on $\bd \DD$. Thus, we need only check that the arc $\zeta(R^P_\theta)$ lands. To see this, note that the map 
$$
\xi := \zeta \circ \Phi^{-1} : \Phi(\Om_0) \to \zeta(\Om_0) 
$$
is quasiconformal with $|\xi(z)| \to 1$ as $|z| \to 1$, so it extends to a homeomorphism of the unit circle. Since $\Phi(R^P_\theta \cap \Om_0)$ lands at $\exp(2 \pi i \Pi(\theta))$ by \thmref{B'}, the arc $\zeta(R^P_\theta \cap \Om_0)$ lands at $\xi(\exp(2 \pi i \Pi(\theta)))$. \vs 

(iii) $\Longrightarrow$ (i): By the above paragraph, the assumption that $\zeta(R^P_\theta)$ and $\zeta(R^P_{\theta'})$ land at the same point implies 
$$
\xi(\exp(2 \pi i \Pi(\theta)))=\xi(\exp(2 \pi i \Pi(\theta'))). 
$$
Since $\xi$ is a homeomorphism of the unit circle, we conclude that $\Pi(\theta)=\Pi(\theta')$. \vs

The last two assertions are straightforward consequences of (i) and (iii), respectively.   
\end{proof}

\section{Proof of \thmref{C}}\label{sec:val}

Let $\Pi: \TT \to \TT$ be the degree $1$ monotone extension of the semiconjugacy between $\whD^{\circ k}|_I$ and $\whd$ constructed in \S \ref{consp}, and $f: \TT \to \TT$ be a degree $d$ monotone extension of $\whD^{\circ k}|_I$ as in \S \ref{uniqq}. We construct a (maximal) Cantor set $C \subset I$ whose gaps are precisely the interiors of the non-degenerate fibers of $\Pi$. We will prove \thmref{C} by bounding the number of points of $I$ that lie in a given gap of $C$.  

\subsection{A Cantor subset of $I$.}

For $x \in \TT$, let $H_x$ denote the fiber $\Pi^{-1}(x)$. The semiconjugacy relation $\Pi \circ f = \whd \circ \Pi$ (\lemref{sem}) and monotonicity of $f$ show that 
$$
f(H_y) = H_x \qquad \text{if} \ x=\whd(y)
$$
and 
\begin{equation}\label{aayy}
f^{-1}(H_x) = \bigcup_{y \in \whd^{-1}(x)} H_y.
\end{equation}
Define 
$$
C := \TT \sm \bigcup_{x \in \TT} \mathring{H}_x. 
$$
Here each $\mathring{H}_x$, the interior of the fiber $H_x$, is an open interval (possibly empty). Evidently $C$ is a non-empty compact proper subset of the circle. 

\begin{lemma}\label{maxC}
$C$ is a Cantor subset of $I$ with $f(C)=C$.   
\end{lemma}

\begin{proof}
Every gap of $I$ is contained in a unique $\mathring{H}_x$ since $\Pi$ is constant on it. This shows $C \subset I$. Moreover, $C$ is totally disconnected since $I$ is, and it has no isolated points since distinct $\mathring{H}_x$'s have disjoint closures. This proves that $C$ is a Cantor set. \vs 

To see the $f$-invariance property of $C$, first suppose $f(\theta) \in \mathring{H}_x$ for some $x$. Then by \eqref{aayy} there is a $y \in \whd^{-1}(x)$ such that $\theta \in \mathring{H}_y$. This proves $f(C) \subset C$. To verify the reverse inclusion, suppose $\theta \in C$ and take any $\theta'$ with $f(\theta')=\theta$. If $\theta' \in C$, then $\theta \in f(C)$ and we are done. Otherwise $\theta' \in \mathring{H}_y$ for some $y$. Setting $x:=\whd(y)$, it follows from $f(H_y)=H_x$ that either $H_x = \{ \theta \}$, or $H_x$ is a non-degenerate closed interval having $\theta$ as a boundary point. In either case, monotonicity of $f$ implies that some endpoint of $H_y$ maps to $\theta$, implying $\theta \in f(C)$. This proves $C \subset f(C)$. 
\end{proof}

\begin{remark}
It is easy to see that $C$ is the maximal Cantor subset of $I$. In fact, if $X$ is any Cantor subset of $I$ not contained in $C$, then some gap of $C$ would contain uncountably many points of $X$. All such points would have the same image under $\Pi$, which would contradict finiteness of the fibers of $\Pi$ in $I$ (\thmref{C}). Alternatively, $C$ can be characterized as the set $I^{\ast}$ of all ``condensation points'' of $I$, that is, the set of all $\theta \in I$ such that every neighborhood of $\theta$ contains uncountably many points of $I$. This follows from the theorem of Cantor-Bendixson according to which $I^{\ast}$ is perfect and $I \sm I^{\ast}$ is at most countable. 
\end{remark}

The above invariance shows that the commutative diagram \eqref{scj} restricts to the Cantor set $C$: 
$$
\begin{tikzcd}[column sep=small]
C \arrow[d,swap,"\Pi"] \arrow[rr,"\whD^{\circ k}"] & & C \arrow[d,"\Pi"] \\
\TT \arrow[rr,"\whd"] & & \TT 
\end{tikzcd} 
$$
Since the closures of gaps of $C$ are precisely the non-degenerate fibers of $\Pi$, it follows that the analog of \lemref{gapmap} holds for $C$, that is, for each gap $]a,b[$ of $C$, either $\whD^{\circ k}(a)=\whD^{\circ k}(b)$ or $]\whD^{\circ k}(a),\whD^{\circ k}(b)[$ is a gap of $C$. \vs   

The degree $d$ extension $f$ of $\whD^{\circ k}|_I$ also serves as an extension of $\whD^{\circ k}|_C$. The above diagram shows that for any gap $J_0$ of $C$, the image $f(\ov{J_0})$ is a single point in $C$ if $J_0$ is taut, and it is the closure of a gap $J_1$ of $C$ if $J_0$ is loose. Note that in the latter case $f$ maps the closed interval $\ov{J_0}$ onto the closed interval $\ov{J_1}$ monotonically, but it may fail to map $J_0$ onto $J_1$ homeomorphically (for example a whole subinterval of $J_0$ may map to an endpoint of $J_1$). In practice, it is convenient to ignore this issue and abuse the language slightly by saying that $J_0$ ``maps to'' $J_1$, or that $J_1$ is the ``image'' of $J_0$. \vs

Now the same argument as in \corref{gapcount} shows that $C$ has $D^k-d$ major gaps counting multiplicities. Evidently each gap of $I$ is contained in a gap of $C$. For each gap $J_i$ of $C$ of multiplicity $n_i$, let $m_i$ be the number of gaps of $I$ contained in $J_i$ counting multiplicities. Since $\sum_i n_i= \sum_i m_i = D^k-d$ and $0 \leq m_i \leq n_i$, we must have $m_i=n_i$ for all $i$. In other words, {\it every major gap of $C$ of multiplicity $n$ contains precisely $n$ major gaps of $I$ counting multiplicities.}  \vs

Finally, let us observe that every minor gap of $C$ eventually maps to a major gap $J$. If $J$ is taut, the next image is a single point and the gap-orbit terminates. If $J$ is loose, the next image is a new gap (minor or major) and the gap-orbit continues. Since there are only finitely many majors gaps, we conclude that {\it every gap of $C$ eventually maps to a taut gap or a periodic gap}.  

\subsection{Taut gaps of $C$}  

\begin{lemma}\label{partners}
Suppose $a,b \in I$ satisfy $\Pi(a)=\Pi(b)$ and $\whD^{\circ k}(a)=\whD^{\circ k}(b)$. Then either $]a,b[$ or $]b,a[$ is a taut gap of $I$. Assuming the former, $R^P_a=R^-_a$ and $R^P_b=R^+_b$ crash into an escaping critical point $\omega$ of $P^{\circ k}$ and their image $P^{\circ k}(R^-_a) = R_{\whD^{\circ k}(a)}=R_{\whD^{\circ k}(b)}=P^{\circ k}(R^+_b)$ is a smooth ray.  
\end{lemma}

Note that ray pairs in $I$ such as the above $R^-_a, R^+_b$ are ``permanent partners'' in the sense that they stay together once they join: $R^-_a(s)=R^+_b(s)$ for all potentials $s \leq G(\omega)$. 

\begin{proof}
Assume by way of contradiction that $R^P_a, R^P_b$ are disjoint. By \corref{samefiber} the distinct points $R^P_a(s),R^P_b(s)$ have bounded hyperbolic distance in $\CC \sm K$ and therefore their Euclidean distance tends to $0$ as $s \to 0$. Let $z \in \bd K$ be any accumulation point of $R^P_a$ and take a sequence of potentials $s_j \to 0$ such that $R^P_a(s_j) \to z$. Then $R^P_b(s_j) \to z$ also. The assumption $\whD^{\circ k}(a)=\whD^{\circ k}(b)$ implies that the points $R^P_a(s_j),R^P_b(s_j)$ have the same image under $P^{\circ k}$, so $P^{\circ k}$ is not locally injective near $z$. This shows that the common accumulation set of $R^P_a, R^P_b$ consists only of critical points of $P^{\circ k}$. Since the accumulation set of a ray is connected and $P^{\circ k}$ has finitely many critical points, it follows that $R^P_a, R^P_b$ co-land at a critical point $c$ of $P^{\circ k}$ on $\bd K$. This implies that $K$ meets both components of $\CC \sm (R^P_a \cup R^P_b \cup \{ c \})$, contradicting \corref{samefiber}. \vs

The preceding paragraph shows that $R^P_a, R^P_b$ crash into a critical point $\omega$ of the Green's function. Assuming $R^P_a = R^-_a$ and $R^P_b=R^+_b$, this implies that $]a,b[$ is a gap of $I$. Moreover, the image rays $P^{\circ k}(R^-_a)$ and $P^{\circ k}(R^+_b)$ both accumulate on $\bd K$ and have the same angle $\whD^{\circ k}(a)=\whD^{\circ k}(b)$. Hence the image rays coincide and are smooth and $\omega$ is in fact a critical point of $P^{\circ k}$. 
\end{proof}

\begin{lemma}\label{tCtI}
Taut gaps of $C$ and all their preimages are gaps of $I$. 
\end{lemma}

\begin{proof}
Let $J = ]a,b[$ be a gap of $C$ which after $n \geq 1$ iterates maps to a taut gap of $C$. For $i \geq 0$ set $a_i:=\whD^{\circ ik}(a), b_i:=\whD^{\circ ik}(b)$ and $J_i:=]a_i, b_i[$, so we have the gap-orbit $J_0 \to J_1 \to \cdots \to J_n$ under $f$, with $J_n$ taut. We will show that $J_i \cap I=\es$ for all $0 \leq i \leq n$. \vs

By \lemref{partners} $R^-_{a_n}, R^+_{b_n}$ form a ray pair in $I$, so $J_n \cap I = \es$. Assume by way of contradiction that there is a largest $0 \leq i \leq n-1$ for which $J_i \cap I \neq \es$. Take $\theta \in J_i \cap I$. Then $f(\theta)=\whD^{\circ k}(\theta) \in \ov{J_{i+1}} \cap I$, so by the choice of $i$, $f(\theta)$ must be one of the endpoints of $J_{i+1}$, say $a_{i+1}$ (the case $f(\theta)=b_{i+1}$ is similar). Since $\whD^{\circ k}(\theta) = \whD^{\circ k}(a_i)=a_{i+1}$ and $\Pi(\theta)=\Pi(a_i)$, \lemref{partners} shows that $R^-_{a_i}, R^+_{\theta}$ form a ray pair in $I$ and their image $R^P_{a_{i+1}}$ is smooth. Applying the iterate $P^{\circ (n-i-1)k}$, it follows that $R^P_{a_n}$ is smooth. This contradicts the fact that $R^P_{a_n}=R^-_{a_n}$ is left broken. 
\end{proof}

\subsection{Periodic gaps of $C$.}\label{pergp}

Let us begin by recalling a known relationship between periodic angles in $I$ and periodic points on the boundary of the component $K$. For $z \in \bd K$, it will be convenient to use the notation 
$$
\La(z):= \{ \theta \in I: R^P_\theta \ \text{lands at} \ z \}.
$$

\begin{theorem}\label{perr}
\mbox{}
\begin{enumerate}
\item[(i)]
If $\theta \in I$ is periodic under $\whD^{\circ k}$, then $\theta \in \La(z_0)$ for some $z_0 \in \bd K$ which is a repelling or parabolic periodic point under $P^{\circ k}$. \vs
\item[(ii)]
Conversely, if $z_0 \in \bd K$ is a repelling or parabolic periodic point under $P^{\circ k}$, then $\La(z_0)$ is non-empty and finite. Moreover, all angles in $\La(z_0)$ are periodic with the same period under $\whD^{\circ k}$.    
\end{enumerate}
\end{theorem}

Observe that in (ii), periodicity of the rays landing at $z_0$ implies that they are either smooth or infinitely broken (in particular, these rays are pairwise disjoint). Thus, each cycle of rays landing at $z_0$ consists either entirely of smooth rays, or entirely of infinitely broken rays.  \vs 

This result is not new: Part (i), the landing of periodic rays on repelling or parabolic points, follows from a classical application of hyperbolic metrics introduced by Sullivan, Douady and Hubbard. For proofs in the case of connected filled Julia sets see for example \cite[Theorem 18.10]{M1} or \cite[Proposition 2.1 and its Complement]{P}. Part (ii) appears in \cite{LP} in a more general setting which also covers the degenerate case $K = \{ z_0 \}$. Here we give an alternative proof based on \thmref{B} and the corresponding statement in the connected case (see also \cite[Corollary B.3]{P}).    

\begin{proof}[Proof of part (ii)]
Let $\varphi$ be a hybrid equivalence between the restriction of $P^{\circ k}$ to a neighborhood of $K$ and a monic degree $d$ polynomial $Q$, and choose the semiconjugacy $\Pi: I \to \TT$ such that \thmref{B} holds. Let $\ell$ be the period of $z_0$ under $P^{\circ k}$. 
Then $w_0:=\varphi(z_0) \in \bd K_Q$ has period $\ell$ under $Q$ and is repelling or parabolic since this property is preserved under topological conjugacies. By \cite[Corollary B.1]{P}, the set $\La(w_0):=\{ \tau \in \TT: R^Q_\tau \ \text{lands at} \ w_0 \}$ is non-empty and finite. 
By \thmref{B}, $\La(z_0)=\Pi^{-1}(\La(w_0))$ and the following diagram commutes:
\begin{equation}\label{cdlala}
\begin{tikzcd}[column sep=small]
\La(z_0) \arrow[d,swap,"\Pi"] \arrow[rr,"\whD^{\circ k\ell}"] & & \La(z_0) \arrow[d,"\Pi"] \\
\La(w_0) \arrow[rr,"\whd^{\circ \ell}"] & & \La(w_0) 
\end{tikzcd} 
\end{equation}  
Since $\La(z_0)$ is compact and invariant under the expanding map $\whD^{\circ k\ell}$, we conclude that $\La(z_0)$ must also be finite (\cite[Lemma 18.8]{M1}). There is a neighborhood of $z_0$ in which $P^{\circ k\ell}$ is a conformal isomorphism since $(P^{\circ k\ell})'(z_0) \neq 0$. It follows that $P^{\circ k\ell}$ is injective on the set of ``ends'' of the rays that land at $z_0$. Since there are finitely many such ends, they must be permuted by $P^{\circ k\ell}$ and therefore they are all periodic. This of course implies periodicity of the whole rays that land at $z_0$, and proves that each angle in $\La(z_0)$ is periodic under $\whD^{\circ k\ell}$. The fact that $P^{\circ k\ell}$ is a conformal isomorphism near $z_0$ implies that it preserves the cyclic order of ray ends landing at $z_0$, so $\whD^{\circ k\ell}$ preserves the cyclic order of angles in $\La(z_0)$. A standard exercise then shows that all angles in $\La(z_0)$ must have the same period. 
\end{proof} 

We can say a bit more about the correspondence between rays landing at $z_0$ and those landing at $w_0$. The external rays of $Q$ landing at $w_0$ fall into some number $N \geq 1$ of cycles under $Q^{\circ \ell}$ which have the same length $q \geq 1$ and the same {\bit combinatorial rotation number} $p/q$ (compare \cite[Corollary B.1]{P}). This means that if we label the angles in $\La(w_0)$ as $0 \leq \tau_1 < \tau_1 < \cdots < \tau_{Nq} <1$, then $\whd^{\circ \ell}$ acts on $\La(w_0)$ as $\tau_j \mapsto \tau_{j+Np}$, taking the subscripts modulo $Nq$. Thus, the action of $Q^{\circ \ell}$ on the rays landing at $w_0$ combinatorially mimics that of the rational rotation $z \mapsto e^{2\pi i p/q} z$. \vs

We claim that the cycles of $\whD^{\circ k \ell}$ in $\Lambda(z_0)$ are also of the common length $q$ and combinatorial rotation number $p/q$. To see this, let $\tau, \tau':=\whd^{\circ \ell}(\tau) \in \Lambda(w_0)$ so $\whD^{\circ k\ell}$ maps $\Pi^{-1}(\tau)$ bijectively onto $\Pi^{-1}(\tau')$. Label the elements of these fibers as
$$
\Pi^{-1}(\tau)=\{ \theta_1, \ldots, \theta_{\nu} \} \qquad \text{and} \qquad \Pi^{-1}(\tau')= \{ \theta'_1, \ldots, \theta'_{\nu} \}  
$$ 
in counterclockwise order. Taking the conformal isomorphism $\zeta: \CC \sm K \to \CC \sm \ov{\DD}$, it follows from \corref{samefiber} that the arcs $\zeta(R^P_{\theta_1}), \ldots, \zeta(R^P_{\theta_{\nu}})$ co-land at some $u \in \bd \DD$, and similarly the arcs $\zeta(R^P_{\theta'_1}), \ldots, \zeta(R^P_{\theta'_{\nu}})$ co-land at some $u' \in \bd \DD$. The conjugate map $f:= \zeta \circ P^{\circ k \ell} \circ \zeta^{-1}$ extends by reflection to a conformal isomorphism $f:\Omega \to \Omega'$ between annular neighborhoods of $\bd \DD$ which sends $u$ to $u'$, preserves $\bd \DD$ and maps each arc $\zeta(R^P_{\theta_i}) \cap \Omega$ to some arc $\zeta(R^P_{\theta'_j}) \cap \Omega'$. The fact that $f$ is orientation-preserving then implies that $f(\zeta(R^P_{\theta_i}) \cap \Omega)=\zeta(R^P_{\theta'_i}) \cap \Omega'$, or $P^{\circ k\ell}(R^P_{\theta_i})=R^P_{\theta'_i}$, or $\whD^{\circ k\ell}(\theta_i)=\theta'_i$ for every $1 \leq i \leq \nu$. Applying this $q$ times, we conclude that the iterate $\whD^{\circ k \ell q}$ must act as the identity on the fiber $\Pi^{-1}(\tau)$. It easily follows that each cycle of $\whD^{\circ k \ell}$ in $\Lambda(z_0)$ has length $q$ and combinatorial rotation number $p/q$. \vs

We summarize the above observations in the following 

\begin{corollary}
Let $z_0 \in \bd K$ be a repelling or parabolic point of period $\ell$ under $P^{\circ k}$ and $w_0=\varphi(z_0) \in \bd K_Q$ be the corresponding periodic point of $Q$. Let $N \geq 1$ be the number of cycles of $\whd^{\circ \ell}$ in $\La(w_0)$ with the common length $q \geq 1$ and combinatorial rotation number $p/q$. Take representative angles $\tau_1, \ldots, \tau_N$ for these cycles and let $\nu_j$ be the number of elements in $\Pi^{-1}(\tau_j)$. Then, there are precisely $\sum_{j=1}^N \nu_j$ cycles of $\whD^{\circ k\ell}$ in $\La(z_0)$ and they all have the same length $q$ and combinatorial rotation number $p/q$. 
\end{corollary}

We now return to our discussion of the periodic gaps of the maximal Cantor set $C$. Suppose $]a,b[$ is a periodic gap of $C$. Then $a,b$ are periodic under $\whD^{\circ k}$ so by \thmref{perr}(i) and \thmref{B} the rays $R^P_a, R^P_b$ co-land at a periodic point $z_0 \in \bd K$. Another application of \thmref{B} then shows that for every $\theta \in ]a,b[ \, \cap I$ the ray $R^P_\theta$ lands at the same $z_0$. Using the notations in the above corollary, it follows that $[a,b] \cap I$ is a fiber of $\Pi$ in $\Lambda(z_0)$, so $\# ([a,b] \cap I) = \nu_j$ for some $1 \leq j \leq N$. Thus, to bound the cardinality of $[a,b] \cap I$, we need to find an upper bound for the $\nu_j$. \vs

It is known that the number of distinct cycles of external rays that land on a periodic orbit of a polynomial is at most one more than the number of critical values (see \cite{M1} for the quadratic and \cite{K} for the higher degree case\footnote{They prove this bound for smooth rays, but their argument works almost verbatim for the infinitely broken periodic rays since they are all pairwise disjoint.}). This gives the estimate $\sum_{j=1}^N \nu_j \leq D$, which implies $\nu_j \leq D$ for all $1 \leq j \leq N$. However, we need a sharper form of this bound for the proof of \thmref{C}. \vs

Each of the iterated images $K^i:=P^{\circ i}(K)$ is a period $k$ component of the filled Julia set $K_P$. \thmref{A} applied to each $K^i$ gives a compact set $I^i \subset \TT$ consisting of angles $\theta$ such that $R_\theta$ or one of $R^{\pm}_{\theta}$ accumulates on $\bd K^i$, and a surjection $\Pi^i: I^i \to \TT$ that semiconjugates $\whD^{\circ k}|_{I^i}$ to $\whd$. For $\theta \in I^i$, we denote by $R^P_\theta$ the unique external ray of $P$ at angle $\theta$ that accumulates on $\bd K^i$. With the periodic point $z_0 \in \bd K^0$ as above, consider its orbit ${\mathcal O}:= \{ z_0, z_1, \ldots, z_{k \ell -1} \}$ under $P$, so $z_i \in K^i$. For each $0 \leq i \leq k \ell-1$ the $q \sum_{j=1}^N \nu_j$ external rays landing at $z_i$ separate the plane into the same number of open topological disks called the {\bit sectors} based at $z_i$. Two distinct sectors are either disjoint or nested or they contain each other's complements. This is a simple consequence of the fact that distinct rays landing on $\mathcal O$, being smooth or infinitely broken, cannot intersect. Moreover, if a sector contains $z_i$, it contains all but one of the sectors based at $z_i$. The collection of all sectors based at the points of $\mathcal O$ will be denoted by ${\mathcal S}_{\mathcal O}$. \vs 

Let $S(z_i,a,b) \in {\mathcal S}_{\mathcal O}$ denote the sector based at $z_i$ bounded by the rays $R_a^P, R_b^P$, labeled counterclockwise so $S(z_i,a,b)$ contains all field lines $R_\theta$ with $\theta \in ]a,b[$. The map $\whD$ acting on the union $\Lambda(z_0) \cup \cdots \cup \Lambda(z_{k \ell-1})$ induces a {\bit sector map} $\sigma: {\mathcal S}_{\mathcal O} \to {\mathcal S}_{\mathcal O}$ defined by $\sigma(S(z_i, a,b)):= S(z_{i+1},\whD(a),\whD(b))$. Locally near the base points, $\sigma$ is compatible with $P$: There are neighborhoods $U$ of $z_i$ and $V$ of $z_{i+1}$ such that $P:U \to V$ is a conformal isomorphism with $P(S \cap U)=\sigma(S) \cap V$ for every sector $S$ based at $z_i$. Globally $P(S)$ is often different from $\sigma(S)$, but by the monodromy theorem if $\sigma(S)$ does not contain any critical value of $P$, then $S$ does not contain any critical point of $P$, and $P|_S: S \to \sigma(S)$ is a conformal isomorphism. \vs

The external rays that bound a sector can contain critical points of $P$ when they are infinitely broken. The following convention clarifies when such boundary points should be thought of as belonging to the sector. We say that a critical point $\omega$ of $P$ is {\bit attached} to the sector $S(z_i,a,b)$ if one of the following happens: (i) $\omega \in S(z_i,a,b)$, or (ii) $\omega \in R_a^P$ and $R_a^P=R_a^-$, or (iii) $\omega \in R_b^P$ and $R_b^P = R_b^+$. Similarly, we say that a critical value $v$ of $P$ is attached to $S(z_i,a,b)$ if either (i) $v \in S(z_i,a,b)$, or (ii) $v=P(\omega)$ for some critical point $\omega$ attached to the sector $\sigma^{-1}(S(z_i,a,b))$. Thus, a critical value attached to a sector $S$ is either an interior critical value which may or may not have come from an interior critical point of $\sigma^{-1}(S)$, or a boundary critical value coming from a boundary critical point attached to $\sigma^{-1}(S)$.  \vs

We define the {\bit length} of $S=S(z_i,a,b)$ by $|S|:=b-a$ and its {\bit weight} $w(S)$ as the integer part of $D \, |S|$. An application of the argument principle (see \cite{GM} as well as \cite[Theorem 5.10]{Z}) shows that 
\begin{align*}
w(S) = & \ \text{number of critical points of} \ P \ \text{attached to} \ S \  \\
& \ \text{counting multiplicities}.
\end{align*}
The relation 
$$
|\sigma(S)| = D \, |S| - w(S)
$$
shows that if $|\sigma(S)| \leq |S|$, then there must be a critical point of $P$ attached to $S$ and therefore a critical value of $P$ attached to $\sigma(S)$. In particular, {\it the shortest sector in each cycle of $\sigma$ has a critical value of $P$ attached to it.} \vs

We call $S(z_i,a,b)$ an {\bit essential sector} if $\Pi^i(a) \neq \Pi^i(b)$ and a {\bit ghost sector} if $\Pi^i(a)=\Pi^i(b)$. By \corref{samefiber}, $S(z_i,a,b) \cap K^i \neq \es$ if $S(z_i,a,b)$ is an essential sector, while $S(z_i,a,b) \cap K^i= \es$ and $]a,b[$ is a gap of $I^i$ if $S(z_i,a,b)$ is a ghost sector. It follows that $\sigma$ preserves the type of a sector, i.e., $\sigma(S)$ is a ghost sector if and only if $S$ is. \vs

The ghost sectors based at $z_i$ do not meet $K^i$ but they may well contain a component $K^j$ for some $j \not \equiv i \ (\operatorname{mod} k)$. We call a ghost sector {\bit minimal} if it does not meet the union $K^0 \cup \cdots \cup K^{k-1}$. Equivalently, if it does not contain any point of the orbit $\mathcal O$. 

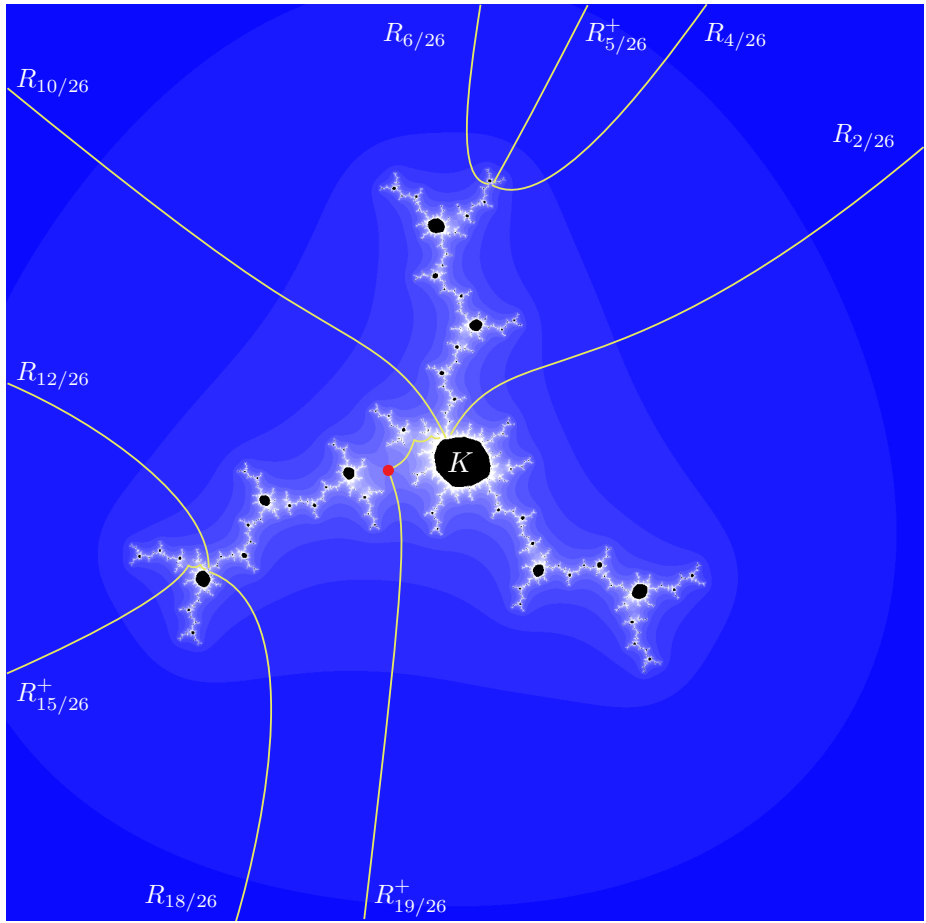
\begin{figure}[t!]
\centering
\begin{overpic}[width=0.8\textwidth]{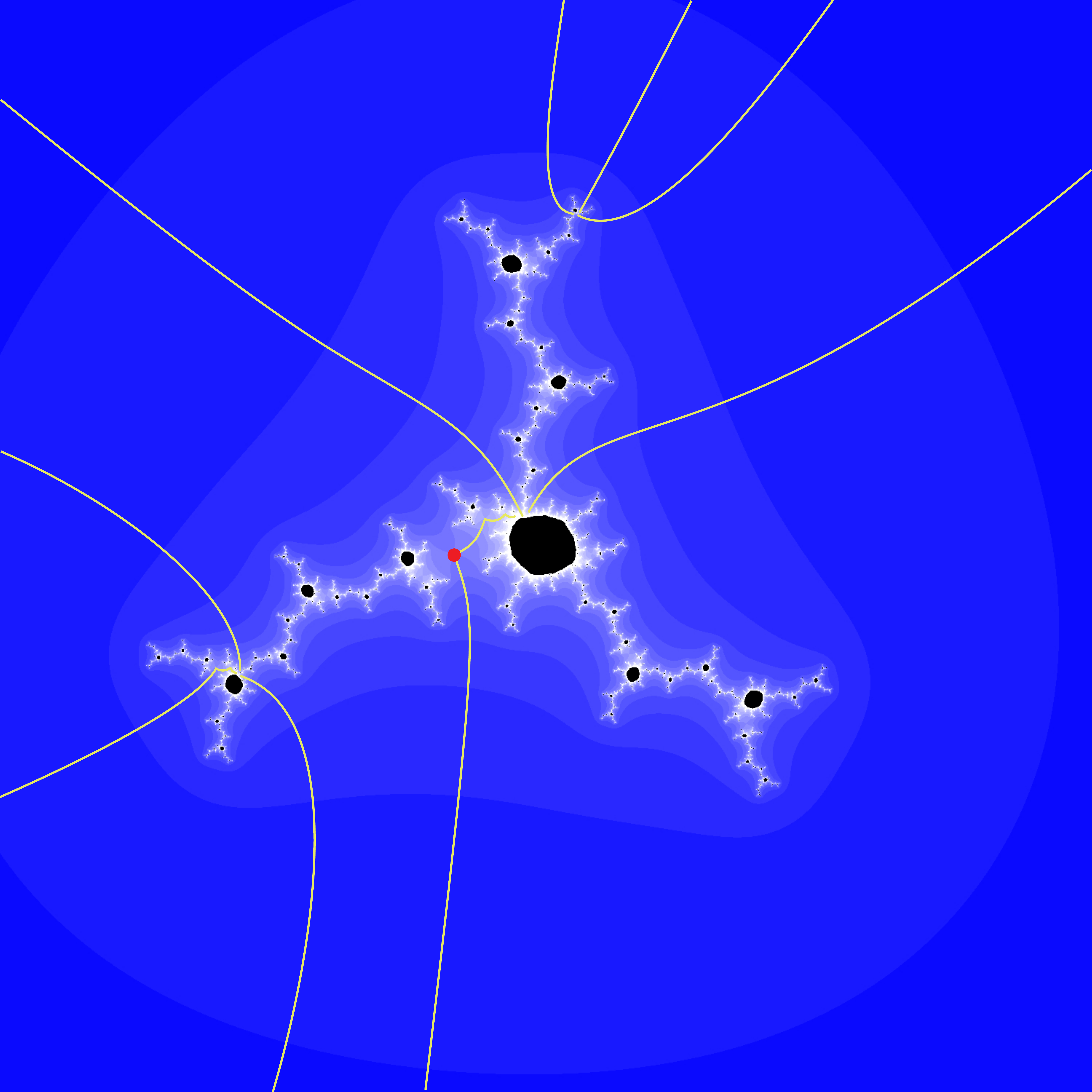}
\put (90,85) {\footnotesize {\color{white} $R_{2/26}$}}
\put (76,96) {\footnotesize {\color{white} $R_{4/26}$}}
\put (63,96) {\footnotesize {\color{white} $R^+_{5/26}$}}
\put (41,96) {\footnotesize {\color{white} $R_{6/26}$}}
\put (1,91) {\footnotesize {\color{white} $R_{10/26}$}}
\put (1,59) {\footnotesize {\color{white} $R_{12/26}$}}
\put (1,24) {\footnotesize {\color{white} $R^+_{15/26}$}}
\put (15,2) {\footnotesize {\color{white} $R_{18/26}$}}
\put (40,2) {\footnotesize {\color{white} $R^+_{19/26}$}}
\put (48,49) {\small {\color{white} $K$}}
\end{overpic}
\caption{\sl{The nine external rays and sectors associated with the period $3$ orbit of the cubic polynomial described in Example \ref{cubicex}. The escaping critical point is shown as a red dot.}} 
\label{ghost1}   
\end{figure}

\begin{example}\label{cubicex}
There is a cubic polynomial with a period $3$ critical point in a component $K$ of the filled Julia set and a period $3$ repelling fixed point $z_0 \in \bd K$ with 
$$
\Lambda(z_0)= \Big\{ \frac{2}{26}, \frac{10}{26}, \frac{19}{26} \Big\}, \quad 
\Lambda(z_1)= \Big\{ \frac{4}{26}, \frac{5}{26}, \frac{6}{26} \Big\}, \quad 
\Lambda(z_2)= \Big\{ \frac{12}{26}, \frac{15}{26}, \frac{18}{26} \Big\}
$$ 
(compare Figures \ref{ghost1} and \ref{ghost2}). Of the nine external rays landing on the orbit ${\mathcal O}=\{ z_0, z_1, z_2 \}$, the three rays $R^+_{19/26}$ (crashing into the escaping critical point), $R^+_{5/26}, R^+_{15/26}$ are infinitely broken and the remaining six are smooth. Of the nine sectors in ${\mathcal S}_{\mathcal O}$, the three sectors
$$
S\Big(z_0, \frac{19}{26}, \frac{2}{26}\Big), \quad S\Big(z_1, \frac{5}{26}, \frac{6}{26}\Big), \quad S\Big(z_2, \frac{15}{26}, \frac{18}{26}\Big)
$$
are essential. The remaining six are ghost sectors, with  
$$
S\Big(z_1, \frac{4}{26}, \frac{5}{26}\Big) \quad \text{and} \quad S\Big(z_2, \frac{12}{26}, \frac{15}{26}\Big)
$$ 
being minimal. 
\end{example}

\begin{figure}[t!]
\centering
\begin{overpic}[width=0.8\textwidth]{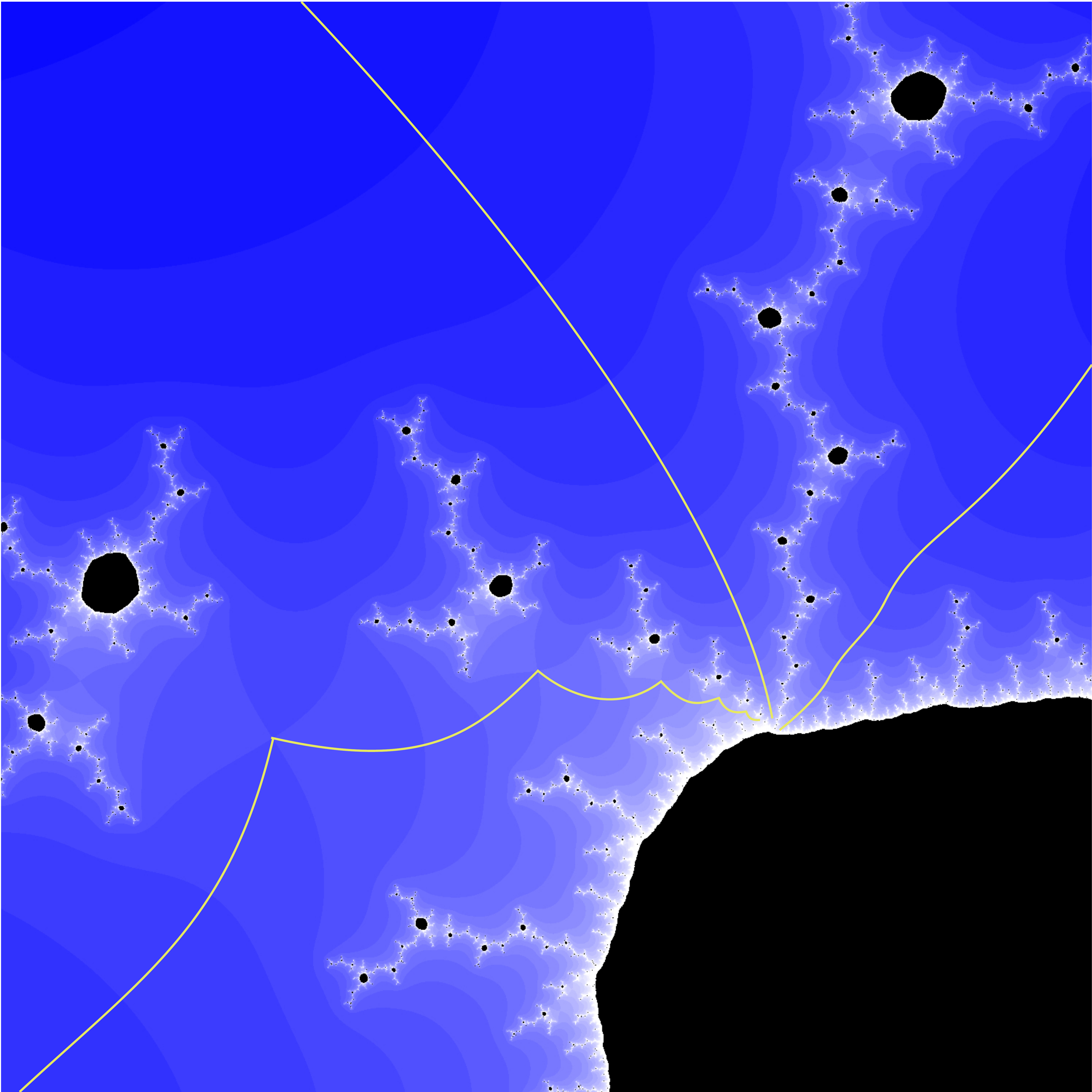}
\put (89,66) {\small {\color{white} $R_{2/26}$}}
\put (33,96) {\small {\color{white} $R_{10/26}$}}
\put (10,4) {\small {\color{white} $R^+_{19/26}$}}
\put (69,30) {\color{white} $z_0$}
\put (80,15) {\color{white} $K$}
\end{overpic}
\caption{\sl{Details of \figref{ghost1} near the period $3$ point $z_0 \in \bd K$.}} 
\label{ghost2}   
\end{figure}

There is a one-to-one correspondence between cycles of $\sigma$ in ${\mathcal S}_{\mathcal O}$ and cycles of $\whD$ in the union $\Lambda(z_0) \cup \cdots \cup \Lambda(z_{k\ell-1})$ (for example, assign to the $\sigma$-cycle of $S(z_i,a,b)$ the $\whD$-cycle of $a$). It follows from our earlier discussion that there are $N$ cycles of essential sectors and $\sum_{j=1}^N \nu_j -N = \sum_{j=1}^N (\nu_j-1)$ cycles of ghost sectors in ${\mathcal S}_{\mathcal O}$. Let 
\begin{align*}
N_1 := & \ \text{number of critical values of} \ P \ \text{attached to some minimal} \\
& \ \text{ghost sector in} \ {\mathcal S}_{\mathcal O}, \\  
N_2 := & \ \text{number of escaping critical values of} \ P \ \text{not attached to} \\
& \ \text{any minimal ghost sector in} \ {\mathcal S}_{\mathcal O}. 
\end{align*}
Note that each critical value that contributes to the count $N_1$ is attached to a {\it unique} minimal ghost sector. Furthermore, the critical values that contribute to the count $N_1$ or $N_2$ can only come from the $D-d$ critical points of $P$ outside $K^0 \cup \cdots \cup K^{k-1}$. Hence,
\begin{equation}\label{n1n2}
N_1+N_2 \leq D-d. 
\end{equation}

\begin{lemma}\label{NaN}
The number of cycles of ghost sectors in ${\mathcal S}_{\mathcal O}$ is $\leq N_1+1$. The inequality is strict if $K$ has period $k=1$. 
\end{lemma}

\begin{proof}
The following argument is inspired by \cite[Theorem 5.2]{K}. Let $M$ denote the number of cycles of ghost sectors in ${\mathcal S}_{\mathcal O}$. There is nothing to prove if $M=1$, so let us assume $M \geq 2$. The shortest sector in each of these $M$ cycles has a critical value attached to it. Let $S_1, \ldots, S_M$ denote these shortest sectors labeled so that $|S_1| \leq \cdots \leq |S_M|$. We prove that $S_1, \ldots, S_{M-1}$ are minimal. This gives the bound $M-1 \leq N_1$, as required. \vs

If $S_i$ is not minimal for some $1 \leq i \leq M-1$, it contains a point $z \in {\mathcal O}$ and therefore it contains all but one of the sectors based at $z$. Among these sectors there must be at least one representative from each of the $M-1$ cycles of ghost sectors other than the cycle represented by $S_i$ itself. This implies that there are $M-1$ ghost sectors shorter than $S_i$, a contradiction. \vs

When $K$ has period $1$, all ghost sectors are automatically minimal. Thus the shortest sectors $S_1, \ldots, S_M$ are minimal, giving the improved bound $M \leq N_1$.        
\end{proof}

\begin{corollary}\label{cardper}
If $]a,b[$ is a periodic gap of $C$ with $a,b \in \Lambda(z_0)$, then $\# ([a,b] \cap I) \leq N_1+2$. The inequality is strict if $K$ has period $k=1$. 
\end{corollary}

\begin{proof}
By what we have seen, $\#([a,b] \cap I) = \nu_i$ for some $1 \leq i \leq N$. By \lemref{NaN}, $\sum_{j=1}^N (\nu_j-1) \leq N_1+1$ so $\nu_i \leq N_1+2$. If $K$ has period $1$, the above sum is bounded by $N_1$ so $\nu_i \leq N_1+1$. 
\end{proof}

\subsection{Preperiodic gaps of $C$}

Let $]a,b[$ be a strictly preperiodic gap of $C$. For $i \geq 0$ set $a_i:=\whD^{\circ ik}(a), b_i:=\whD^{\circ ik}(b)$, so each $]a_i,b_i[$ is a gap of $C$. Let $n \geq 1$ be the smallest integer for which $]a_n,b_n[$ is periodic. Then $R^P_{a_n}, R^P_{b_n}$ co-land at a periodic point $z_0 \in \bd K$. As in \S \ref{pergp}, let $\mathcal O$ denote the orbit of $z_0$ under $P$ and let ${\mathcal S}_{\mathcal O}$ be the collection of sectors based at the points of $\mathcal O$. Recall that $N_2 \geq 0$ is the number of escaping critical values of $P$ that are not attached to any minimal ghost sector in ${\mathcal S}_{\mathcal O}$.   

\begin{lemma}\label{prepcount}
$\# ([a,b] \cap I) \leq \# ([a_n,b_n] \cap I) + N_2$.    
\end{lemma}

\begin{proof}
Under the iterate $\whD^{\circ nk}$ every angle in $[a,b] \cap I$ maps to $[a_n,b_n] \cap I$. Moreover, by \lemref{partners}, for each $\theta_n \in [a_n,b_n] \cap I$ there can be at most two angles in $[a,b] \cap I$ which map to $\theta_n$ under $\whD^{\circ nk}$. Let us assume $a \leq \theta<\theta' \leq b$ are such angles and set $\theta_i:= \whD^{\circ ik}(\theta), \theta'_i:=\whD^{\circ ik}(\theta')$, so $\theta_n=\theta'_n$. Let $0 \leq i \leq n-1$ be the largest integer for which $\theta_i \neq \theta_i$. By \lemref{partners}, the rays $R^-_{\theta_i}, R^+_{\theta'_i}$ crash into a critical point $\omega$ of $P^{\circ k}$ and the common image $P^{\circ k}(R^-_{\theta_i})=R_{\theta_{i+1}}=R_{\theta'_{i+1}}=P^{\circ k}(R^+_{\theta'_i})$ is smooth. Hence there is an integer $1 \leq j \leq k$ for which $v:=P^{\circ j}(\omega)$ is an escaping critical value of $P$. If $v$ were attached to a minimal ghost sector $S \in {\mathcal S}_{\mathcal O}$, it would necessarily be an interior critical value of $S$ since it has the smooth or finitely broken ray $R:=P^{\circ j}(R^-_{\theta_i})$ passing through it. This would imply that the entire ray $R$ is contained in $S$, and therefore the landing point $\zeta$ of $R$ belongs to $\ov{S} \cap K^j$. Since $S$ does not meet the component $K^j$, the point $\zeta$ would have to be the base point of $S$, and $R$ would be one of the rays bounding $S$, which is a contradiction. \vs

We have assigned to each such pair $\theta, \theta'$ at least one escaping critical value of $P$ not attached to any minimal ghost sector in ${\mathcal S}_{\mathcal O}$. Evidently different pairs have different critical values assigned to them, so the number of such pairs is at most $N_2$. The lemma follows immediately.       
\end{proof}

We are now ready to finish the proof of \thmref{C}: 

\begin{proof}[Proof of \thmref{C}]
Let $]a,b[$ be a gap of the Cantor set $C$. We need to show that the closed interval $[a,b]$ contains at most $D-d+2$ points of $I$. Set $a_i:=\whD^{\circ ik}(a), b_i:= \whD^{\circ ik}(b)$ for $i \geq 0$. We consider two cases: \vs

$\bullet$ {\it Case 1.} There is a smallest $n \geq 1$ such that $a_n=b_n$. Then $]a_{n-1}, b_{n-1}[$ is a taut gap of $C$, so $]a,b[$ is a gap of $I$ by \lemref{tCtI}. In this case $[a,b]$ contains exactly two points of $I$, i.e. the endpoints $a,b$. \vs 

$\bullet$ {\it Case 2.} $a_i \neq b_i$ and therefore $]a_i, b_i[$ is a gap of $C$ for all $i \geq 0$. Then there is a smallest $n \geq 0$ such that $]a_n,b_n[$ is periodic. By \corref{cardper}, $[a_n,b_n]$ contains at most $N_1+2$ points of $I$. It follows from \lemref{prepcount} and the inequality \eqref{n1n2} that the cardinality of $[a,b] \cap I$ is at most $N_1+N_2+2 \leq D-d+2$, as required. Again by \corref{cardper} this inequality is strict if $K$ has period $1$. 
\end{proof}        

\section{Proof of \thmref{D}}\label{sec:ex}

In this section we give a descriptive proof of \thmref{D} by constructing examples of polynomials $P$ of degree $D \geq 3$ having a filled Julia set component $K$ of period $k=1$ and polynomial-like degree $d \geq 2$ for which the semiconjugacy $\Pi: I \to \TT$ of \thmref{A} has the top valence $D-d+1$ asserted by \thmref{C}. If $D-d+1 \geq 3$, it follows that $I$ has isolated points. The common feature of these examples is that $D-d+1$ consecutive fixed rays of $P$ co-land at a fixed point on $\bd K$. Our construction is flexible in that we can designate the hybrid class of the restriction of $P$ to a neighborhood of $K$ as well as the fixed rays that co-land on $\bd K$.  

\subsection{General description of the examples}\label{descex}

The fixed rays of a monic degree $D$ polynomial have angles $\theta_i := i/(D-1) \modd$, taking the subscript $i$ modulo $D-1$. Fix an integer $j$ and choose a collection of $D-d$ open intervals of the form 
$$
J_i= \Big] \theta_i, \theta_i + \frac{1}{D} \Big[ \quad \text{or}  \quad \Big] \theta_i-\frac{1}{D}, \theta_i \Big[
$$
subject to the conditions that (i) each $J_i$ is contained in the interval $]\theta_j, \theta_{j+D-d}[$, and (ii) the $J_i$ have pairwise disjoint closures. (The reader can verify that there are $D-d+1$ choices for such collections). Each of the $D-d$ intervals $]\theta_i, \theta_{i+1}[$ for $j \leq i \leq j+D-d-1$ contains exactly one element of $\{ J_i \}$, namely $J_i = ]\theta_i,\theta_i+1/D[$ or $J_{i+1}=]\theta_{i+1}-1/D,\theta_{i+1}[$. For simplicity let $\theta'_i:=\theta_i \pm 1/D$ denote the endpoint of $J_i$ other than $\theta_i$. Note that $\whD(\theta_i)=\whD(\theta'_i)=\theta_i$. \vs 

Take a polynomial $Q$ of degree $d$ with $K_Q$ connected. Let $P$ be a monic polynomial of degree $D$ with the following properties: \vs 
\begin{enumerate}
\item[(i)]
There is a component $K=P(K)$ of $K_P$ and neighborhoods $U_0, U_1$ of $K$ such that $P|_{U_1}: U_1 \to U_0$ is a polynomial-like map of degree $d$ hybrid equivalent to $Q$. \vs
\item[(ii)]
The $D-d$ critical points of $P$ that do not belong to $K$ are distinct and escape to $\infty$. \vs
\item[(iii)]
For each $J_i$ in the chosen collection, the field lines $R_{\theta_i}$ and $R_{\theta'_i}$ crash into an escaping critical point $\omega_i$. \vs
\end{enumerate}
We claim that these properties imply that the angles $\theta_j, \ldots, \theta_{j+D-d}$ are in $I=I_K$ and belong to the same fiber of the semiconjugacy $\Pi:I \to \TT$. By \thmref{B} and \thmref{perr} the corresponding rays $R^P_{\theta_j}, \ldots, R^P_{\theta_{j+D-d}}$ would co-land at a repelling or parabolic fixed point on $\bd K$. This will reduce the proof of \thmref{D} to the construction of a polynomial $P$ satisfying (i)-(iii). \vs 

By (iii) the union $R_{\theta_i} \cup R_{\theta'_i} \cup \{ \omega_i \}$ bounds a ``wake'' $W_i$ containing all field lines $R_\theta$ with $\theta \in J_i$. Note that $P$ maps $W_i$ univalently onto the domain $\C \sm R_{\theta_i}([Ds_i,+\infty[)$ which properly contains $W_i$. Here $s_i>0$ is the Green's potential of $\omega_i$. It follows from the Schwarz lemma that $W_i$ contains a unique fixed point $p_i$ which is necessarily repelling. In particular, this shows that $K$ is not contained in $W_i$, so $K \cap W_i = \es$. It is easy to see that one of the two broken rays $R_{\theta_i}^\pm$ (more specifically, $R_{\theta_i}^+$ if $J_i= ]\theta_i , \theta'_i [$ or $R_{\theta_i}^-$ if $J_i= ]\theta'_i , \theta_i [$) lands at $p_i$, whereas the other lands at a fixed point $z_i$ that lies outside of the union $\bigcup W_j$. Of the $D$ fixed points of $P$ counting multiplicities, $D-d$ are $\{ p_i \}$ that fall in $\bigcup W_j$. The remaining $d$ fixed points must be in $K$ since $K \subset \CC \sm \bigcup W_j$. This shows that every $z_i$ belongs to $K$ and therefore $\theta_i \in I$. To prove our claim, we show that $]\theta_i, \theta_{i+1}[ \, \cap I = \es$ for all $j \leq i \leq j+D-d-1$ (it will of course follow that the $z_i$ are one and the same point). \vs

\begin{figure}[t]
\centering
\begin{overpic}[width=0.95\textwidth]{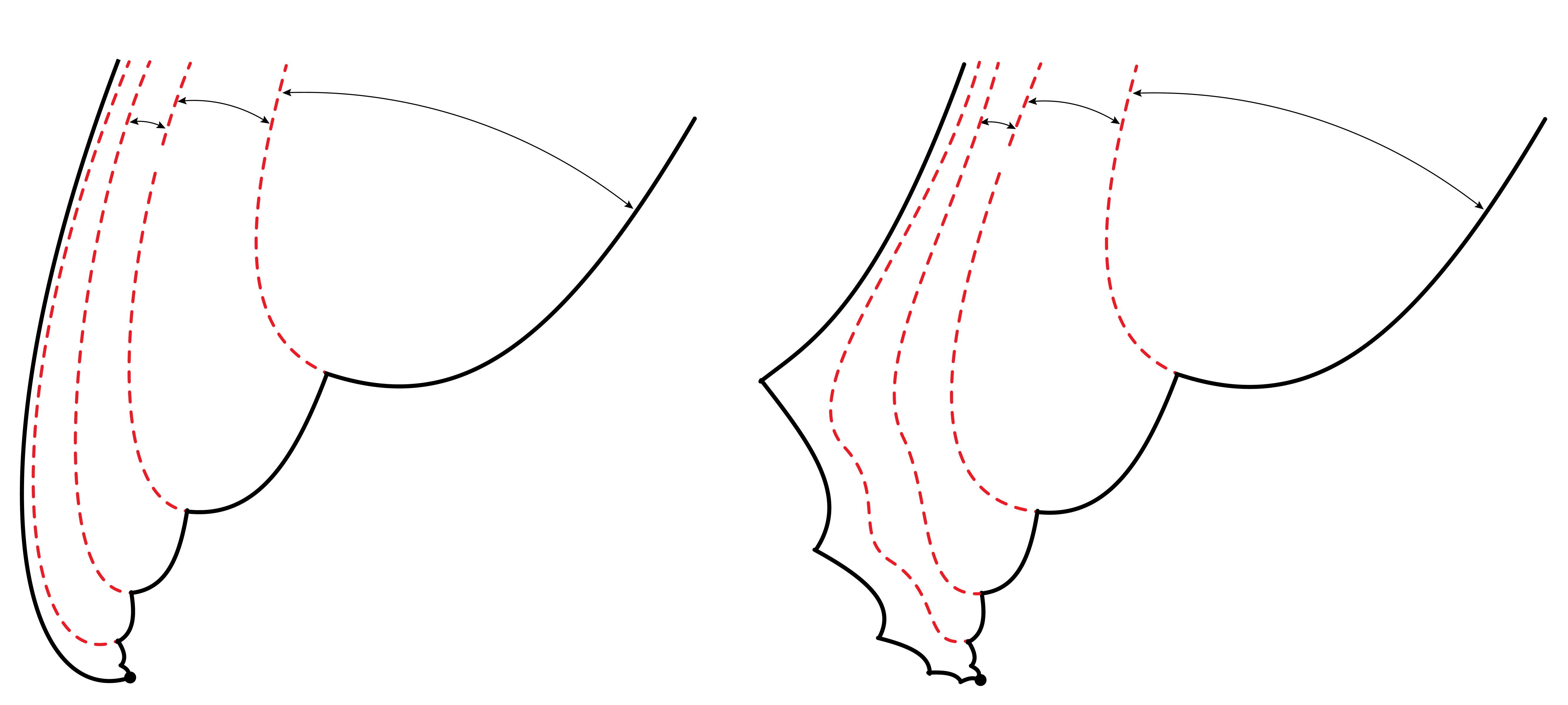}
\put (43,39) {\small $\alpha_0$}
\put (17,42) {\color{red}{\small $\alpha_1$}}
\put (12,42) {\color{red}{\small $\alpha_2$}}
\put (8,42) {\color{red}{\small $\alpha_3$}}
\put (-5,40) {\small $\alpha_0+\tfrac{1}{D-1}$}
\put (27,34) {\footnotesize $D^{-1}$}
\put (12,35.7) {\tiny $D^{-2}$}
\put (7.5,34) {\tiny $D^{-3}$}
\put (20.5,22) {\small $c_1=\omega_i$}
\put (11,13.5) {\small $c_2$}
\put (7.2,8.3) {\small $c_3$}
\put (99,39) {\small $\alpha_0$}
\put (72,42) {\color{red}{\small $\alpha_1$}}
\put (67,42) {\color{red}{\small $\alpha_2$}}
\put (63,42) {\color{red}{\small $\alpha_3$}}
\put (49,40) {\small $\alpha_0+\tfrac{1}{D-1}$}
\put (81,34) {\footnotesize $D^{-1}$}
\put (66,35.7) {\tiny $D^{-2}$}
\put (61.8,34) {\tiny $D^{-3}$}
\put (74.5,22) {\small $c_1=\omega_i$}
\put (65.2,13.5) {\small $c_2$}
\put (61.5,8.3) {\small $c_3$}
\end{overpic}
\caption{\sl The preimages of the interval $]\alpha_0,\alpha_1[$ exhaust the interval $]\alpha_0,\alpha_0+1/(D-1)[$ between two consecutive fixed points of $\whD$. The fixed ray at angle $\alpha_0+1/(D-1)$ can be either smooth (left) or infinitely broken (right).}  
\label{wakes}
\end{figure}
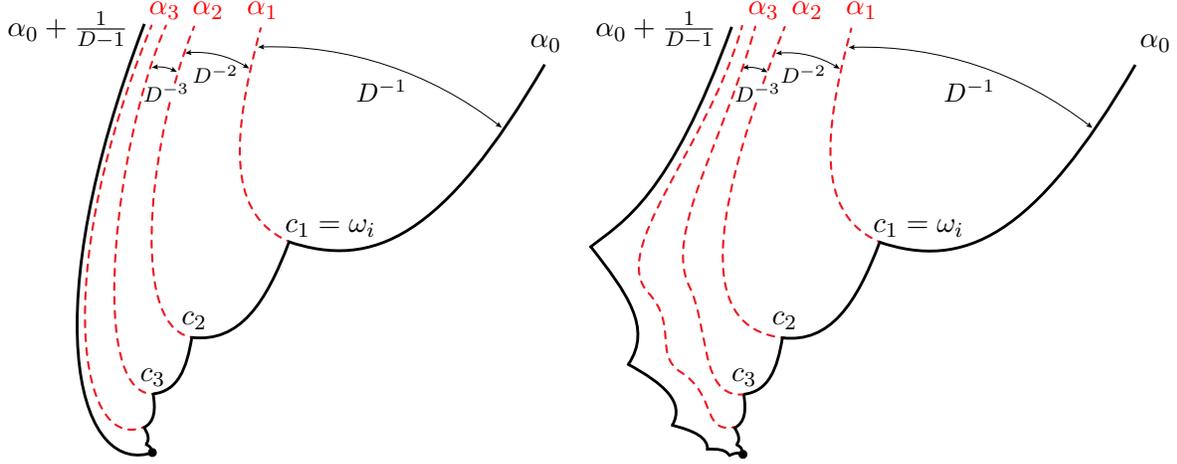

Suppose $]\theta_i, \theta_{i+1}[$ contains $J_i = ]\theta_i, \theta'_i[ $ (the case where it contains $J_{i+1} = ]\theta'_{i+1}, \theta_{i+1}[$ is similar). Set $\alpha_0:=\theta_i, \alpha_1:=\theta'_i=\alpha_0+1/D$. By (ii) the rays $R^-_{\alpha_0}, R^+_{\alpha_1}$ first crash into the critical point $\omega_i=R^-_{\alpha_0}(s_{\alpha_0})$. Since $\whD(\alpha_0)=\alpha_0$, the ray $R^-_{\alpha_0}$ is infinitely broken at the preimages $c_n := R^-_{\alpha_0}(s_{\alpha_0}/D^{n-1})$ of $\omega_i$ (compare \S \ref{gr}). Thus, for each $n \geq 1$ there is an angle $\alpha_n \in ]\alpha_0, \alpha_0+1/(D-1)[$ such that the rays $R^-_{\alpha_0}, R^+_{\alpha_n}$ crash into $c_n$ (see \figref{wakes}). The relation $P(c_n)=c_{n-1}$ gives $\whD(\alpha_n)=\alpha_{n-1}$ which shows  
$$
\alpha_n=\alpha_0+\frac{1}{D}+\frac{1}{D^2}+\cdots+\frac{1}{D^n}. 
$$
Thus, $\alpha_n \to \alpha_0+1/(D-1)$ as $n \to \infty$. Since $]\alpha_0,\alpha_n[ \, \cap I = \es$, we conclude that $]\alpha_0,\alpha_0+1/(D-1)[ \, \cap I = \es$, as required. \vs

\begin{figure}[t!]
\centering
\begin{overpic}[width=0.8\textwidth]{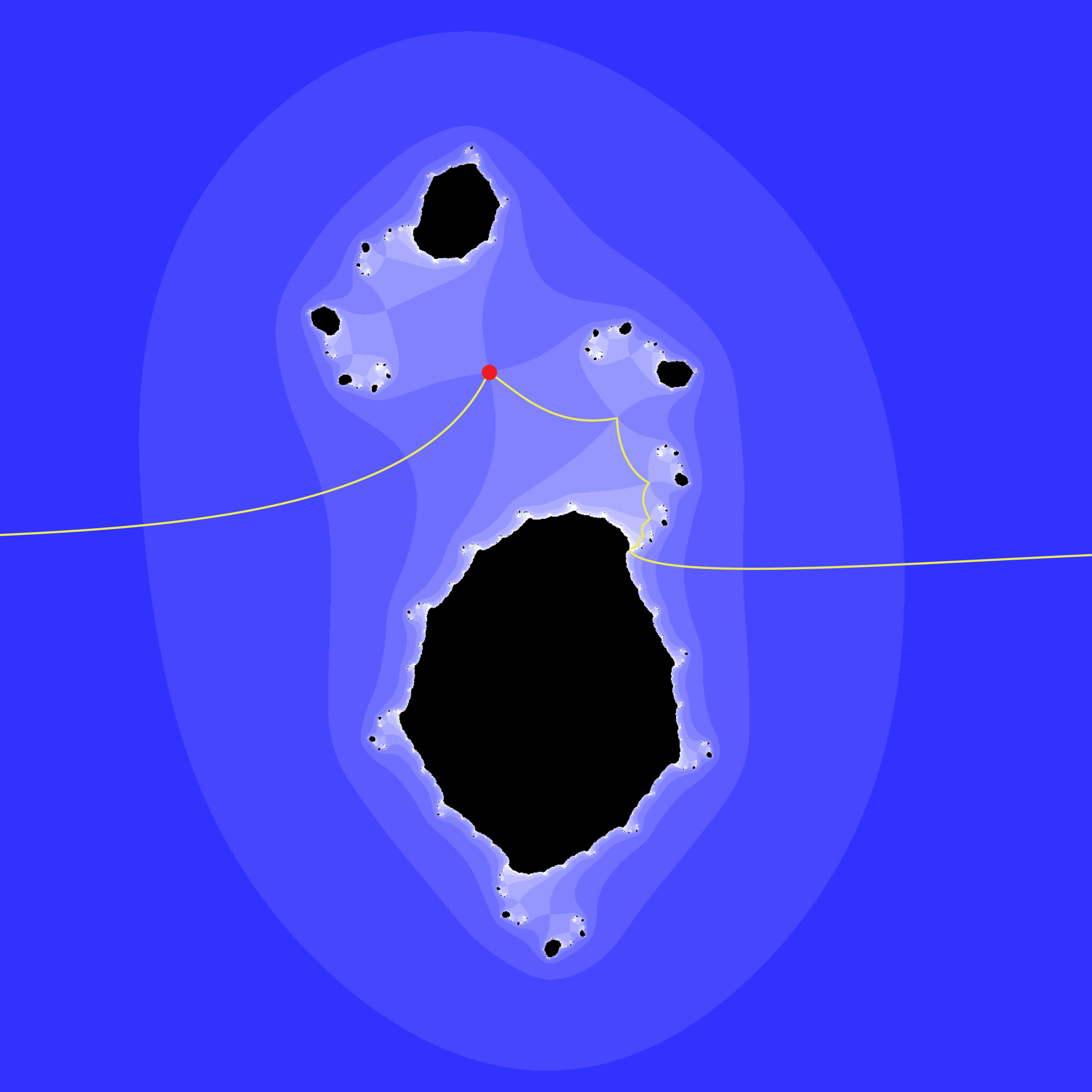}
\put (94,51) {\small {\color{white} $R_0$}}
\put (1,47) {\small {\color{white} $R^+_{1/2}$}}
\put (54,47) {\small {\color{white} $\beta$}}
\put (44.5,67) {\small {\color{white} $\omega_1$}}
\put (48,36) {{\color{white} $K$}}
\end{overpic}
\caption{\sl{Filled Julia set of a degree $D=3$ polynomial $P$ with a quadratic-like restriction hybrid equivalent to $Q(z)=z^2$. The fixed rays $R_0, R^+_{1/2}$ co-land at the fixed point $\beta$ on the boundary of the component $K$ of $K_P$. Here $P(z)=a z^2+z^3$ with $a \approx 0.31629-i 1.92522$.}} 
\label{seh}   
\end{figure}

\begin{figure}[t!]
\centering
\begin{overpic}[width=0.8\textwidth]{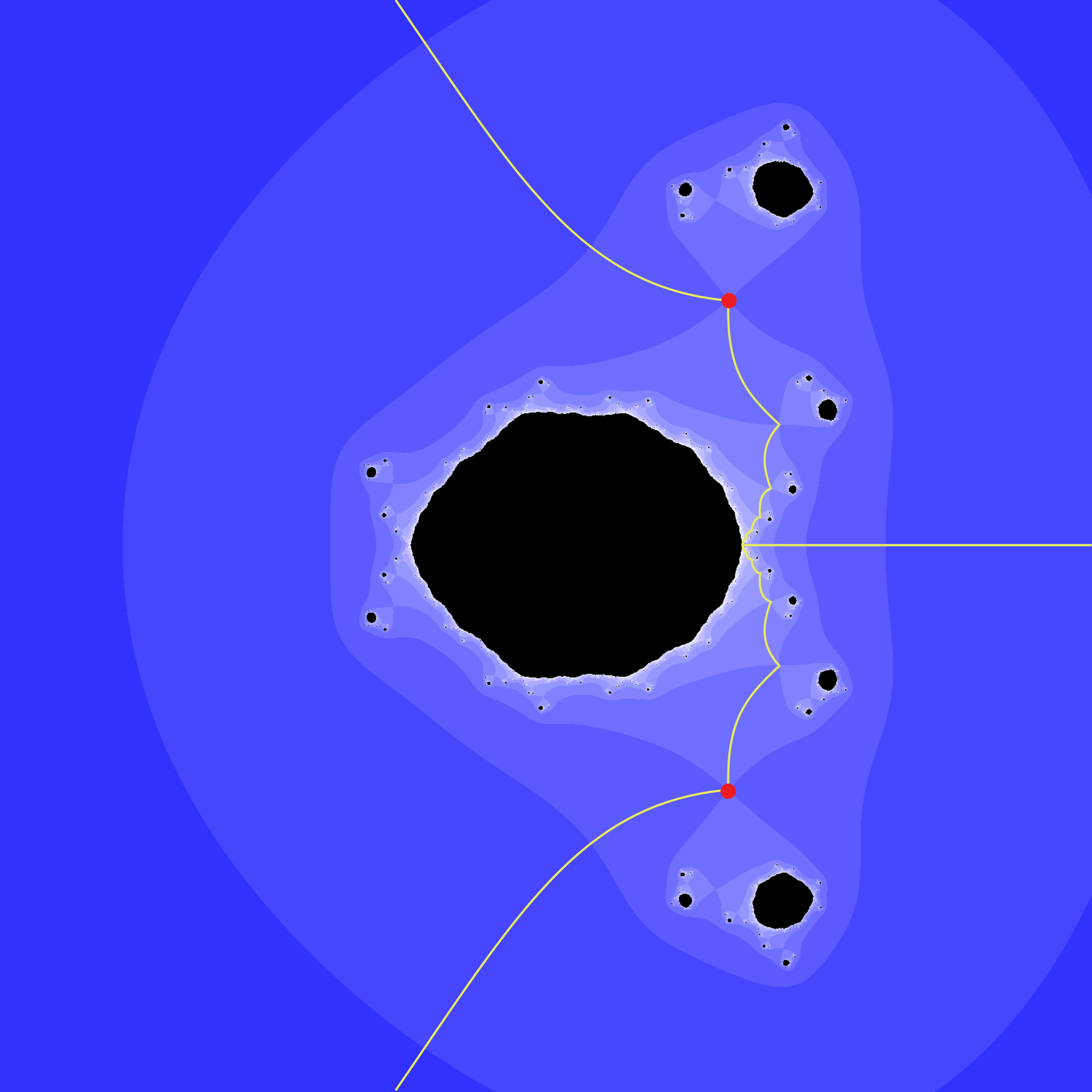}
\put (94,52) {\small {\color{white} $R_0$}}
\put (29,96) {\small {\color{white} $R^+_{1/3}$}}
\put (28,3) {\small {\color{white} $R^-_{2/3}$}}
\put (64.3,49) {\small {\color{white} $\beta$}}
\put (68.2,72) {\small {\color{white} $\omega_1$}}
\put (68.2,27) {\small {\color{white} $\omega_2=\ov{\omega_1}$}}
\put (51,49) {{\color{white} $K$}}
\end{overpic}
\caption{\sl{Filled Julia set of a degree $D=4$ polynomial $P$ with a quadratic-like restriction hybrid equivalent to $Q(z)=z^2$. The fixed rays $R_0, R^+_{1/3}, R^-_{2/3}$ co-land at the fixed point $\beta$ on the boundary of the component $K$ of $K_P$. Here $P(z)=\sqrt[3]{10} \, z^2+a z^3+z^4$ with $a \approx -1.64846$.}} 
\label{chah}   
\end{figure}

Figures \ref{seh} and \ref{chah} illustrate examples of cubic and quartic polynomials of this type with super-attracting fixed points at $0$, so their quadratic-like restrictions are hybrid equivalent to $z \mapsto z^2$. \figref{seh} shows the case $D=3$ with the choice 
$$
J_1= \, ] \theta'_1, \theta_1 [ \, = \Big] \frac{1}{6}, \frac{1}{2} \Big[,
$$
where the fixed rays $R_0, R^+_{1/2}$ co-land at a repelling fixed point $\beta \in \bd K$. \figref{chah} shows the case $D=4$ with the choice
$$
J_1 = \, ] \theta'_1, \theta_1 [ \, = \Big] \frac{1}{12}, \frac{1}{3} \Big[ \quad \text{and} \quad J_2 = \, ] \theta_2, \theta'_2 [ \, = \Big] \frac{2}{3}, \frac{11}{12} \Big[, 
$$
where the fixed rays $R_0, R^+_{1/3}, R^-_{2/3}$ co-land at a repelling fixed point $\beta \in \bd K$. In this example $\theta_0=0$ is an isolated point of $I$.    

\subsection{Construction of examples by surgery}

We now give the details of the construction of polynomials that have the properties (i)-(iii) above. The idea is to use a cut-and-paste surgery to construct a synthetic model for such a polynomial, and then apply the measurable Riemann mapping theorem to realize it as an actual polynomial map. To simplify our exposition, we illustrate the construction of a degree $D=6$ polynomial with a degree $d=2$ polynomial-like restriction and $D-d+1=5$ fixed rays (one smooth, four broken) co-landing at a fixed point on $\bd K$. The general case is a straightforward modification of this example. \vs  

For $0 \leq i \leq 4$ let $\theta_i := i/5 \modd$, so each $\theta_i$ is fixed under multiplication by $6 \modd$. Set 
\begin{align*} 
& \theta'_1 := \theta_1-\frac{1}{6}=\frac{1}{30} & & \theta'_2 := \theta_2-\frac{1}{6} = \frac{7}{30}  \\[5pt]
& \theta'_3 := \theta_3+\frac{1}{6}=\frac{23}{30} & & \theta'_4 := \theta_4+\frac{1}{6} = \frac{29}{30}.
\end{align*} 
Fix a radius $R>1$. Define a Riemann surface $\XX$ conformally equivalent to the Riemann sphere $\Chat$ as follows. Cut eight slits in $\Chat$ along each of the closed straight line segments $[0, R\e^{2\pi i \theta}]$ where 
$\theta \in \{ \theta_1, \ldots , \theta_4, \theta'_1, \ldots , \theta'_4 \}$. We can think of the interior of each slit as having two sides which we denote by $\delta_\theta^-$ and $\delta_\theta^+$ for $\theta$ in the above set. In other words, for $0<r<R$,  
\begin{align*}
\delta^+_\theta(r) & := \lim_{\tau \searrow \theta} \ r \e^{2\pi i \tau} \\
\delta^-_\theta(r) & := \lim_{\tau \nearrow \theta} \ r \e^{2\pi i \tau}.
\end{align*}
The union $\YY$ of the slit sphere together with the sixteen arcs $\delta_\theta^\pm$ is a Riemann surface with real-analytic boundary arcs. Define a Riemann surface $\XX^*$ by making the identifications 
$$
\delta_{\theta_i'}^+(r) \longleftrightarrow \delta_{\theta_i}^-(r) \quad \text{and} \quad \delta_{\theta_i'}^-(r) \longleftrightarrow \delta_{\theta_i}^+(r) \quad (0<r<R) 
$$
on the boundary arcs of $\YY$ for $1 \leq i \leq 4$. It is not hard to check that $\XX^*$ is homeomorphic to a $2$-sphere with nine points removed, and that these missing points are analytically punctures. These punctures correspond to the four points 
$$
c_i :=  \lim_{r \to R} \delta_{\theta_i'}^\pm(r)= \lim_{r \to R} \delta_{\theta_i}^\mp(r) \qquad (1 \leq i \leq 4) 
$$
together with the four points 
\begin{align*}
z_1 & :=  \lim_{r \to 0} \delta_{\theta_1'}^+(r)= \lim_{r \to 0} \delta_{\theta_1}^-(r)  & & & z_2 & := \lim_{r \to 0} \delta_{\theta_2'}^+(r)= \lim_{r \to 0} \delta_{\theta_2}^-(r) \\
z_3 & :=  \lim_{r \to 0} \delta_{\theta_3'}^-(r)= \lim_{r \to 0} \delta_{\theta_3}^+(r)  & & & z_4 & := \lim_{r \to 0} \delta_{\theta_4'}^-(r)= \lim_{r \to 0} \delta_{\theta_4}^+(r)
\end{align*}
together with the single point 
\begin{align*}
z_0 & :=  \lim_{r \to 0} \delta_{\theta_1'}^-(r)= \lim_{r \to 0} \delta_{\theta_1}^+(r) = \lim_{r \to 0} \delta_{\theta_2'}^-(r)= \lim_{r \to 0} \delta_{\theta_2}^+(r) \\
& = \lim_{r \to 0} \delta_{\theta_3'}^+(r)= \lim_{r \to 0} \delta_{\theta_3}^-(r) = \lim_{r \to 0} \delta_{\theta_4'}^+(r)= \lim_{r \to 0} \delta_{\theta_4}^-(r) 
\end{align*}
(compare \figref{pants}). If we add these nine punctures to $\XX^\ast$, we obtain a compact Riemann surface $\XX$ homeomorphic to a $2$-sphere and therefore biholomorphic to the Riemann sphere $\Chat$ by the uniformization theorem. The function $\log|z|$ on $\YY$ induces a well-defined subharmonic function $G$ on $\XX$ which tends to $-\infty$ at $z_0, \ldots, z_4$ and takes the value $s_0 := \log R$ at $c_1, \ldots, c_4$ having simple critical points there. For $s \in \RR$ define 
$$
\Om(s) : = \{ z \in \XX : G(z) > s \}.
$$ 
Evidently the identity map gives rise to a biholomorphism 
$$
\zeta : \{ z \in \YY: |z|>R \} \to \Omega(s_0). \vs
$$

The map $z \mapsto z^6$ on $\YY$ induces a degree $6$ holomorphic branched covering $f : \Om(s_0/6) \to \Om(s_0)$ with nine critical points, four simple critical points at $c_1, \ldots, c_4$ and a critical point of multiplicity $5$ at $\infty$. However it is not possible to holomorphically extend $f$ to $\Om(s)$ for any $s < s_0/6$, because each of the pre-images of the $c_i$ will become a point of discontinuity. To circumvent this problem, we use quasiconformal surgery to extend the restriction $f|_{\Om(s_0/3)}$ to a degree $6$ smooth branched covering of $\XX$ with an invariant conformal structure of bounded dilatation. \vs

\begin{figure}[t!]
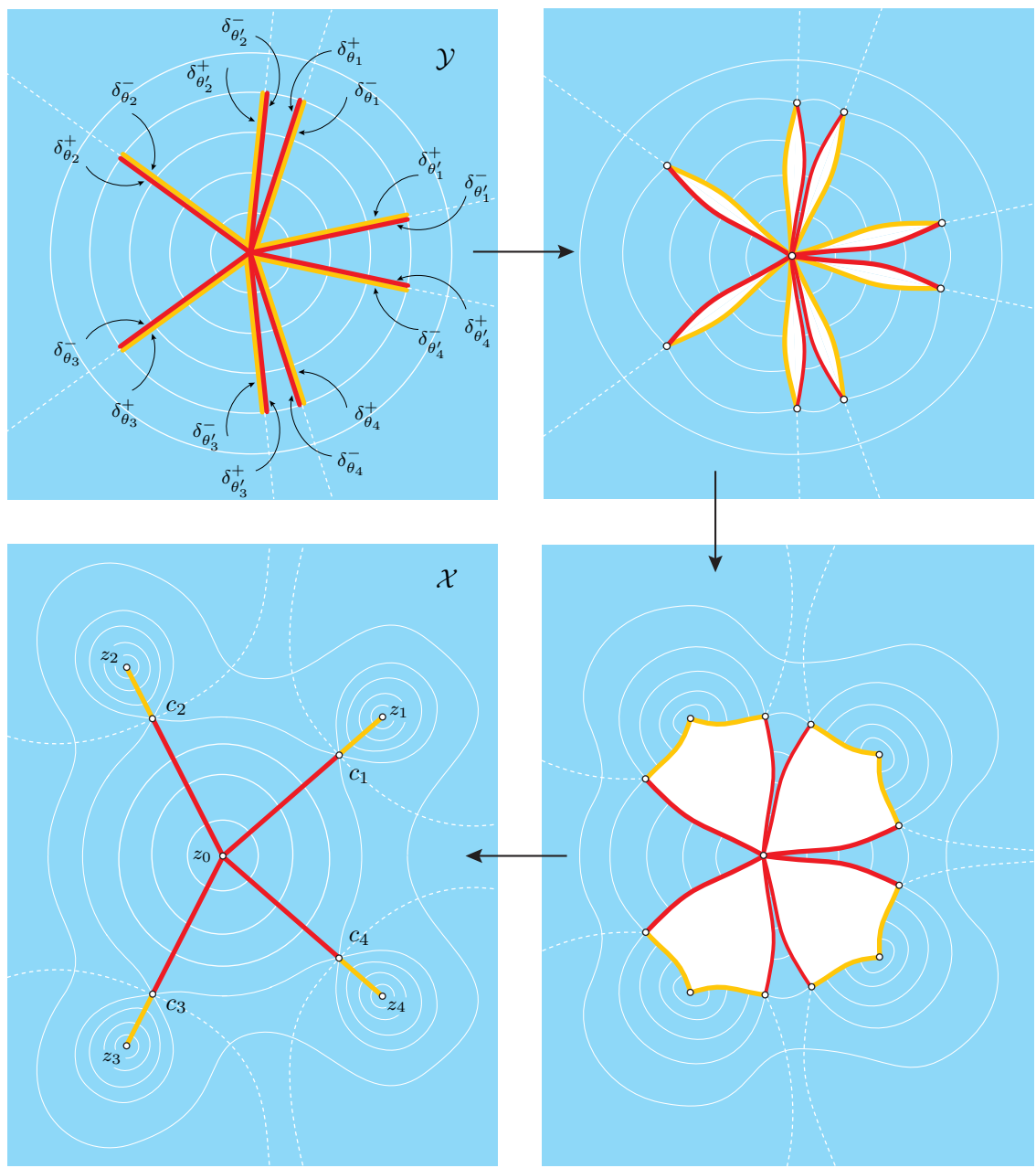

\centering
\begin{overpic}[width=0.98\textwidth]{pants.pdf}
\put (37,95) {\small $\YY$}
\put (37,50) {\small $\XX$}
\put (16,26.5) {\tiny {$z_0$}}
\put (33,39) {\tiny {$z_1$}}
\put (8,43.7) {\tiny {$z_2$}}
\put (8.2,9) {\tiny {$z_3$}}
\put (32.8,13.5) {\tiny{$z_4$}}
\put (29.5,33.4) {\footnotesize {$c_1$}}
\put (13.7,39.3) {\footnotesize {$c_2$}}
\put (13.7,13.5) {\footnotesize {$c_3$}}
\put (29.5,19.4) {\footnotesize {$c_4$}}
\put (39.5,84.5) {\tiny {$\delta_{\theta'_1}^-$}}
\put (39.5,72) {\tiny {$\delta_{\theta'_4}^+$}}
\put (35.5,86.5) {\tiny {$\delta_{\theta'_1}^+$}}
\put (35.5,71) {\tiny {$\delta_{\theta'_4}^-$}}
\put (30,92.5) {\tiny {$\delta_{\theta_1}^-$}}
\put (30,64.5) {\tiny {$\delta_{\theta_4}^+$}}
\put (28.5,96) {\tiny {$\delta_{\theta_1}^+$}}
\put (28.5,60.5) {\tiny {$\delta_{\theta_4}^-$}}
\put (18.5,98) {\tiny {$\delta_{\theta'_2}^-$}}
\put (18.5,59) {\tiny {$\delta_{\theta'_3}^+$}}
\put (15.5,94) {\tiny {$\delta_{\theta'_2}^+$}}
\put (16,63) {\tiny {$\delta_{\theta'_3}^-$}}
\put (9,92.5) {\tiny {$\delta_{\theta_2}^-$}}
\put (9,64.7) {\tiny {$\delta_{\theta_3}^+$}}
\put (4,87.5) {\tiny {$\delta_{\theta_2}^+$}}
\put (4,70) {\tiny {$\delta_{\theta_3}^-$}}
\end{overpic}
\caption{\sl{The cut-and-past construction of the Riemann surface $\XX \cong \Chat$. The solid white curves are the level sets of the function $\log |z|$ on $\YY$ and the induced subharmonic function $G$ on $\XX$. The dashed white curves are the radial lines at angles $\theta_i, \theta'_i$ in $\YY$ and their corresponding images in $\XX$.}} 
\label{pants}   
\end{figure}

For $0 \leq i \leq 4$ and $s \leq s_0$ denote by $D_i(s)$ the disk neighborhood of $z_i$ consisting of points $z$ with $-\infty \leq G(z)<s$. For $0 \leq i \leq 4$ and $s < s_0 < s'$ denote by $A_i(s, s')$ the closed topological annulus $\XX \sm (\Om(s') \cup D_i(s))$. Take any quadratic polynomial $Q$ with connected filled Julia set. For $s>0$ let $V(s)$ denote the topological disk $\{ z \in \C : G_Q(z) < s \}$, where $G_Q$ is the Green's function of $Q$. Fix some $\rho>0$ and let $\phi : V(2\rho) \to D_0(s_0/3)$ be any conformal isomorphism which sends $0$ to $z_0$. Set $D'_0:= \phi(V(\rho))$. The conjugate map 
$$
f_0 := \phi \circ Q \circ \phi^{-1}: D'_0 \to D_0(s_0/3)
$$ 
is a quadratic-like map conformally conjugate to $Q$ which extends to a smooth (in fact real-analytic) degree $2$ covering map between the boundary curves, that is, between the inner boundaries of the closed annuli $\ov{D_0}(s_0/3) \sm D'_0$ and $A_0(s_0/3,2s_0)$. Since $f$ restricts to a smooth degree $2$ covering map between the outer boundaries of the same annuli, we can interpolate between $f_0$ and $f$ to obtain a smooth degree $2$ covering map $\ov{D_0}(s_0/3) \sm D'_0 \to A_0(s_0/3,2s_0)$. This gives a smooth extension of $f$ to $D_0(s_0/3)$ which is holomorphic in $D'_0$. \vs

We can define similar degree $1$ extensions of $f$ to the four remaining components of $\{ z \in \XX : -\infty \leq G(z) < s_0/3 \}$, namely the topological disks $D_i(s_0/3)$ for $1 \leq i \leq 4$. That is, we can find a disk $D'_i$ compactly contained in $D_i(s_0/3)$ and a smooth extension of $f$ to $D_i(s_0/3)$ which maps $D'_i$ conformally onto $D_i(s_0/3)$ and the closed annulus $\ov{D_i}(s_0/3) \sm D'_i$ diffeomorphically onto $A_i(s_0/3,2s_0)$. The details are straightforward and will be left to the reader. \vs

Let $\tilde{f}: \XX \to \XX$ denote the extension of $f|_{\Om(s_0/3)}$ constructed this way. Then $\tilde{f}$ is a degree $6$ branched covering of $\XX$ which is smooth and therefore quasiregular. Define a conformal structure $\mu$ on $\XX$ by setting $\mu=\mu_0$ (the standard conformal structure) on $\Om(s_0/3)$ and extending it by pulling back under the iterates of $\tilde{f}$. In other words, for each $n \geq 1$, set $\mu = (\tilde{f}^{\circ n})^{\ast}(\mu_0)$ in $\tilde{f}^{-n}(\Om(s_0/3))$, and define $\mu=\mu_0$ on $\XX \sm \bigcup_{n \geq 0} \tilde{f}^{-n}(\Om(s_0/3))$. Evidently $\mu$ is $\tilde{f}$-invariant and has bounded dilatation since each backward orbit starting in $\Om(s_0/3)$ passes through each non-holomorphic region $\ov{D_i}(s_0/3) \sm D'_i$ of $\tilde{f}$ once (or twice if it hits the boundary). By the measurable Riemann mapping theorem, there exists a quasiconformal homeomorphism $\psi: \XX \to \Chat$ which pulls back the standard conformal structure of $\Chat$ to $\mu$. We can normalize $\psi$ such that the conformal map $\psi \circ \zeta: \{ z \in \YY : |z|>R \} \to \psi(\Om(s_0))$ is tangent to the identity at $\infty$. The conjugate map $P := \psi \circ \tilde{f} \circ \psi^{-1}$ is then a degree $6$ monic polynomial. Moreover, the conformal map $\frak{B}:=\zeta^{-1} \circ \psi^{-1}: \psi(\Om(s_0)) \to \{ z \in \YY : |z|>R \}$ conjugates $P$ to $z \mapsto z^6$ and is tangent to the identity at $\infty$, so $\frak{B}$ must be the B\"{o}ttcher coordinate for $P$. In particular, each radial line $\{ r \e^{2\pi i \theta} : r>R \}$ in $\YY$ pulls back under $\frak{B}$ to the field line $R_\theta$ for $P$. \vs

Setting $U_1:=\psi(D'_0)$ and $U_0:=\psi(D_0(s_0/3))$, we see that the restriction $P|_{U_1}: U_1 \to U_0$ is a quadratic-like map hybrid equivalent to $Q$ and therefore its filled Julia set $K$ is connected. The boundary of $U_0$ is contained in the basin of $\infty$ and so is the boundary of the topological disk $(P|_{U_1})^{-n}(U_0)$ for every $n \geq 1$. It follows that $K=\bigcap_{n \geq 0} (P|_{U_1})^{-n}(U_0)$ is a connected component of the filled Julia set $K_P$. Moreover, by the construction the four critical points $\omega_i := \psi(c_i)$ of $P$ are escaping at the common potential $s_0$, and the field lines $R_{\theta_i}$ and $ R_{\theta'_i}$ crash into $\omega_i$ for $1 \leq i \leq 4$. Thus $P$ satisfies the conditions (i)-(iii) of \S \ref{descex}. We conclude that the angles $\{ \theta_0, \ldots, \theta_4 \}$ belong to the same fiber of $\Pi: I_K \to \TT$ and the five rays $R_0, R^+_{\theta_1}, R^+_{\theta_2}, R^-_{\theta_3}, R^-_{\theta_4}$ co-land at a fixed point on $\bd K$.

\end{document}